\documentclass[reqno, 12pt]{amsart}
\usepackage[letterpaper,hmargin=1in,vmargin=1in]{geometry}
\usepackage{amsmath,amssymb,amsthm}
\usepackage{verbatim,epsfig,graphics}
\newtheorem{theorem}{Theorem}
\newtheorem{proposition}[theorem]{Proposition}
\newtheorem{corollary}[theorem]{Corollary}
\newtheorem{lemma}[theorem]{Lemma}
\theoremstyle{remark}
\newtheorem{remark}{Remark}
\newtheorem*{definition}{Definition}

\newtheorem*{althyp}{Alternative Hypothesis}
\begin{document}
\title{Conic degeneration and the determinant of the Laplacian}
\author{David A. Sher}
\begin{abstract} We investigate the behavior of various spectral invariants, particularly the determinant of the Laplacian, on a family of smooth Riemannian manifolds $\Omega_\epsilon$ which undergo conic degeneration; that is, which converge in a particular way to a manifold with a conical singularity. Our main result is an asymptotic formula for the determinant up to terms which vanish as $\epsilon$ goes to zero. The proof proceeds in two parts; we study the fine structure of the heat trace on the degenerating manifolds via a parametrix construction, and then use that fine structure to analyze the zeta function and determinant of the Laplacian.
\end{abstract}
\maketitle

\section{Introduction}

Let $N$ 
be a closed Riemannian manifold of dimension $n-1$ with a fixed metric $dy^{2}$. The infinite cone over $N$, which we call $C_N$, is the Riemannian manifold $(0,\infty)_r\times N_y$ with the conic metric
\begin{equation}\label{conicmetric}
dr^{2}+r^{2}dy^{2}.
\end{equation}
Let $\Omega_{0}$ be a compact
manifold of dimension $n\geq 2$ which has a single isolated 
conical singularity
of cross-section $N$ at a point $P\in\Omega_{0}$. This means precisely that there is some $\delta>0$ and some neighborhood of $P$ which is isometric to $(0,\delta]_{r}\times N_{y}$, with the conic metric (\ref{conicmetric}), where $P$ corresponds to $r=0$. Throughout, we assume without loss of generality that $\Omega_{0}$ is exactly conic for $r\leq 2$. Further, let $Z$ be a complete manifold, without boundary, which is exactly conic in a neighborhood of $\infty$ with cross-section $N$; that is, $Z$ has the metric (\ref{conicmetric}), but in a neighborhood of $r=\infty$ rather than $r=0$. Again, without loss of generality, we assume that $Z$ is exactly conic for $r\geq\frac{1}{2}$.

Using $Z$, we construct a degenerating family of smooth manifolds $\Omega_\epsilon$ that converge to $\Omega_{0}$. First, for any fixed $\epsilon\geq 0$, we form a manifold which we call $\epsilon Z$ by scaling the metric on $Z$ by $\epsilon^2$, so that all distances are multiplied by $\epsilon$. If we first cut $Z$ at $r=1/\epsilon$ and then scale the metric by $\epsilon^{2}$, we
obtain the manifold $\epsilon\{Z\cap\{r\leq 1/\epsilon\}\}$. Now for each $\epsilon>0$, we replace the portion of $\Omega_{0}$ where $r\leq 1$ with $\epsilon\{Z\cap\{r\leq 1/\epsilon\}\}$.
The metrics agree in a neighborhood of the boundary $\{r=1\}$, so they may be glued together
to obtain the smooth manifold $\Omega_{\epsilon}$. This process
is illustrated in Figure \ref{construction}.
Note that
in the region $\{r\leq1\}$, $\Omega_{\epsilon}$ is isometric to $\epsilon Z$, while in the region
$\{r\geq\epsilon\}$, $\Omega_{\epsilon}$ is isometric to $\Omega_{0}$. As $\epsilon\rightarrow 0$, the manifolds $\Omega_\epsilon$ converge in a natural way to $\Omega_0$. 

We investigate the behavior of the spectrum of the Laplacian, the trace of the heat kernel, and the determinant of the Laplacian $\Delta_{\Omega_{\epsilon}}$ on $\Omega_\epsilon$ as $\epsilon$ approaches zero. First let $\lambda_{\epsilon,i}$ be the $i$th eigenvalue (counted with multiplicity) of $\Delta_{\Omega_{\epsilon}}$. Further, let $H^{\Omega_{\epsilon}}(t,z,z')$ be the heat kernel on $\Omega_{\epsilon}$ at time $t$, with $Tr H^{\Omega_{\epsilon}}(t)$ its trace. Then recall that the \emph{zeta function} on $\Omega_{\epsilon}$ is
\[\zeta_{\Omega_{\epsilon}}(s)=\sum_{i=1}^{\infty}\lambda_{\epsilon,i}^{-s}=\frac{1}{\Gamma(s)}\int_0^{\infty}(Tr H^{\Omega_{\epsilon}}(t)-1)t^{s-1}\ dt.\]
The zeta function has a meromorphic continuation to $\mathbb C$ with no pole at $s=0$, and the determinant of the Laplacian is given by \[\det\Delta_{\Omega_{\epsilon}}=e^{-\zeta'_{\Omega_{\epsilon}}(0)}.\]

\begin{figure}
\centering
\includegraphics[scale=0.75]{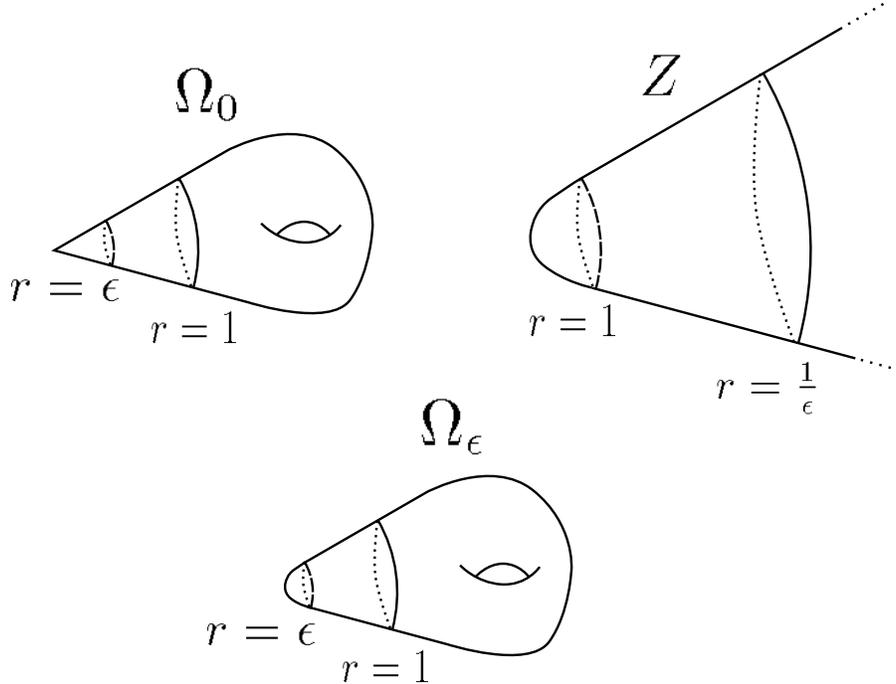}
\caption{Conic degeneration: construction of $\Omega_{\epsilon}$.}\label{construction}
\end{figure}

\subsection{Motivation}

Families of smooth Riemannian manifolds which degenerate to a manifold with conical singularities arise naturally in various settings in differential geometry and PDE.
One such setting is a program, initiated
by Melrose and collaborators, to study elliptic differential equations on singular spaces through
the methods of geometric microlocal analysis. The goal is to study the transition between smooth and singular geometry under the degenerating process, and in particular to examine the limiting behavior of various geometric and index-theoretic invariants. In this setting, conic degeneration was first introduced in the Ph.D. thesis of McDonald \cite{mcd}, who studied the behavior of the Schwarz kernel of the resolvent under a type of conic degeneration similar to the one we consider. The particular technique he used is known as analytic surgery.

Along these lines, one of our ultimate objectives is a generalization of the Cheeger-M\"uller theorem to manifolds with conical singularities.
The Cheeger-M\"uller theorem in the smooth setting concerns the analytic torsion,
which is an invariant
built out of an alternating sum of determinants of Laplacians, first 
defined in \cite{rs}. The analytic torsion appears to be strongly
dependent on the metric; however, the Cheeger-M\"uller theorem shows that
it is in fact equal to a combinatorial invariant called the R-torsion. The
theorem was first conjectured in \cite{rs} 
and then proven independently in \cite{ch1} and \cite{mu}.
Analyzing the determinant of the scalar Laplacian under conic degeneration is a first step towards analyzing the limiting behavior of the analytic torsion. We hope to combine this analysis with an analysis of the R-torsion under conic degeneration to obtain a Cheeger-M\"uller theorem for manifolds with conical singularities. 

Analyzing the behavior of the determinant under conic degeneration may also have
some bearing on the isospectral problem. In 
a well-known series of papers \cite{ops1,ops2}, Osgood,
Phillips, and Sarnak use the determinant of the Laplacian
to prove that any set of isospectral
metrics on a closed surface is compact. A key step in their approach is to show that the log of the determinant is a proper map on the moduli space of
constant curvature metrics. Using similar techniques,
they also prove isospectral
compactness for planar domains in $\mathbb R^{2}$ with
Dirichlet boundary conditions \cite{ops3}.

A natural question, asked by Khuri in \cite{kh}, is whether these results
generalize to sets of isospectral metrics on topological surfaces 
of genus $p$ with $n$ disks removed, again with Dirichlet boundary
conditions, for $np\geq 1$. Khuri showed that the approach of \cite{ops1,ops2,ops3}
cannot work in this setting, because the log of the determinant is no longer a proper map
on the appropriate moduli space. In particular, Khuri
constructed explicit families 
of constant curvature metrics which approach the boundary of
moduli space but which have log determinants staying bounded from below.
These families are constructed by taking a fixed flat metric on a surface of genus $p$ with $n$ conical singularities and then excising smaller and smaller disks around the conic points; see \cite{kh} for the details. This process is similar to our conic degeneration.
We hope that understanding the precise behavior of the determinant under conic degeneration will help us gain a better understanding of the determinant on moduli space and provide
new approaches to the isospectral compactness problem. We should also mention that Khuri's problem was recently partially solved by Kim, who proved isospectral compactness in the $np\geq 1$ case for flat metrics via a different approach \cite{kim}.

\subsection{Prior work}

The particular form of conic degeneration we consider was first introduced in
2004 by Degeratu and Mazzeo \cite{ma2}, in slightly greater generality. Their goal was
to analyze elliptic operators not just on manifolds with conical singularities
but also on iterated cone-edge spaces. They studied the behavior of the
eigenvalues of the Laplacian in this general setting, proving convergence of
the spectrum under certain general assumptions \cite{ma2}.
Our conic degeneration corresponds to the depth-1 case of their
work, for which Rowlett proved the following spectral convergence \cite{ma2,row}:
\begin{theorem}\label{specconv}[Degeratu-Mazzeo, Rowlett] Assume $n\geq 5$.
Let $\lambda_{\epsilon,i}$ be the $i$th eigenvalue
of the Laplacian on $\Omega_{\epsilon}$, and let $\lambda_{0,i}$ be the $i$th
eigenvalue of the Laplacian, with the Friedrichs extension, on $\Omega_{0}$.
Then for each $i$, as $\epsilon\rightarrow 0$,
\[\lambda_{\epsilon,i}\rightarrow\lambda_{0,i}.\]
\end{theorem}
Rowlett also proves this result for $n\geq 3$ under some additional hypotheses
\cite{row2}.
The study of spectral convergence 
has been refined further by Ann\'e and Takahashi in \cite{at}.

It turns out to be significantly harder to analyze the behavior of
more complicated spectral invariants, such as the heat trace and determinant, 
under conic degeneration. For example, based on Theorem \ref{specconv}, one expects that
the heat trace $H^{\Omega_{\epsilon}}(t)$ should
converge as $\epsilon$ goes to zero for any fixed positive $t$. However, this
convergence is not easy to prove, because 
the convergence in Theorem \ref{specconv} is not uniform in $i$. In general, heat trace convergence has only been studied under restrictive curvature assumptions; for example,
if the sectional curvatures of $\Omega_{\epsilon}$ were uniformly bounded below (which is equivalent to the assumption that $Z$ has non-negative sectional curvature), it would follow from \cite{d}. The assumption of non-negative sectional curvature is, however, quite strong, and we wish to avoid it.

Mazzeo and Rowlett, in \cite{mr}, study a closely related problem: that
of degeneration of smoothly bounded planar domains in $\mathbb R^{2}$ to
domains with corners. This is illustrated in Figure \ref{polygons}.
In this setting, with Dirichlet boundary conditions, spectral
convergence holds, and as shown in \cite{mr}, 
it is also true that for each positive $t$,
\begin{equation}\label{htconv}
Tr H^{\Omega_{\epsilon}}(t)\rightarrow Tr H^{\Omega_{0}}(t).
\end{equation}

\begin{figure}\label{polygons}
\centering
\includegraphics[scale=0.75]{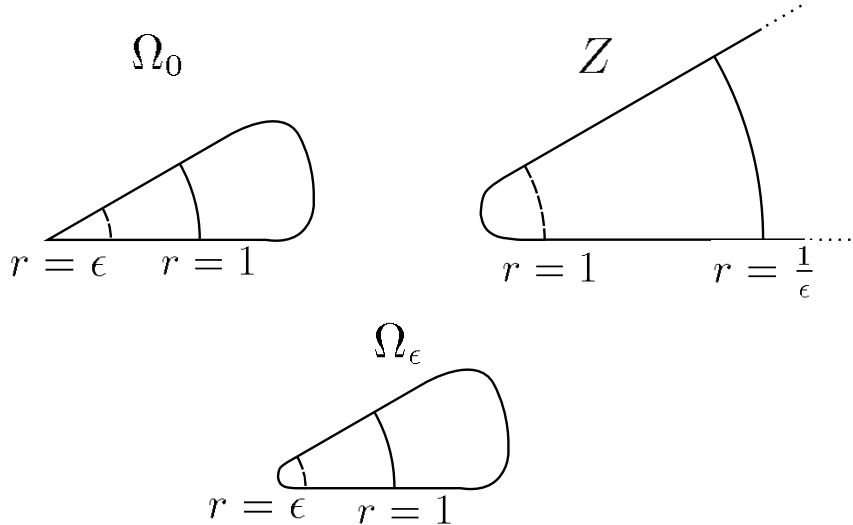}
\caption{Polygonal degeneration of Mazzeo-Rowlett.}
\end{figure}

However, the fine structure of $Tr H^{\Omega_{\epsilon}}(t)$ near
$t=\epsilon=0$ is more complicated; (\ref{htconv}) is not uniform in $t$. 
For each $\epsilon>0$, the usual short-time heat asymptotics imply that
there are $a_{k,\epsilon}$ such that
\[Tr H^{\Omega_{\epsilon}}(t)=a_{0,\epsilon}t^{-1}+a_{1,\epsilon}
t^{-1/2}+a_{2,\epsilon}+\mathcal O(t).\]
A similar expansion holds for $Tr H^{\Omega_{0}}(t)$, with coefficients $a_{k,0}$. As $\epsilon\rightarrow 0$, Mazzeo and Rowlett show that
$a_{0,\epsilon}\rightarrow a_{0,0}$ and $a_{1,\epsilon}\rightarrow a_{1,0}$,
but that
\[a_{2,\epsilon}\rightarrow a_{2,0}+
\sum_{j=1}^{k}(\frac{\pi^{2}-\alpha_{j}^{2}}{24\pi\alpha_{j}}
-\frac{\pi-\alpha_{j}}{12\pi}),\]
where the sum is taken over all interior angles $\alpha_{j}$ of $\Omega_{0}$
\cite{mr}. This anomaly in the heat trace asymptotics was first discovered
by Fedosov \cite{fed} in a context unrelated to degeneration.
Its expression was first simplified by Ray; Ray's work is unpublished,
but a clear explanation may be found in a paper of van den Berg and
Srisatkunarajah \cite{bs}.

Mazzeo and Rowlett analyze the asymptotic structure of 
$Tr H^{\Omega_{\epsilon}}(t)$ for small $\epsilon$ and $t$ and use this structure to explain
the non-uniformity in the heat trace asymptotics. Their work motivates our
approach to analyzing the determinant under conic degeneration; many of the
features in this setting, including the non-uniformity in the asymptotics,
are mirrored in our case.

The behavior of the heat kernel on our degenerating family of manifolds 
has also been studied by Rowlett \cite{row}, who outlines a direct construction for the heat kernel 
on $\Omega_{\epsilon}$ in a slightly more general setting known as 'asymptotically conic convergence.' This construction uses a variant of the
analytic surgery technique of McDonald. Unfortunately, certain details in \cite{row} do not seem to be correct. One lesson from this construction is that the behavior of the full heat kernel under conic degeneration is quite complicated. In the present work, we introduce a less involved approach which focuses only on the heat trace. Since the zeta function only depends on the heat trace and not on the off-diagonal heat kernel, our approach suffices for the analysis of the determinant.

\subsection{Main results}

Many of the proofs in the present work depend on the asymptotic structure of the heat kernel on $Z$, which is analyzed in \cite{s1}. In particular, we prove in \cite{s1} that it is possible to define a renormalized zeta function and determinant of the Laplacian on $Z$, denoted $^{R}\zeta_{Z}$ and $^{R}\det\Delta_{Z}$ respectively. Although $^R\zeta_Z(s)$ may have a pole at $s=0$, we let $^R\zeta_Z(0)$ be the term of order 0 in the Laurent series expansion of $^R\zeta_Z(s)$ at $s=0$; recall from \cite{s1} that $-\log^{R}\det\Delta_Z$ is defined to be the term of order $s$ in the same Laurent series. We then have the following approximation formula for the determinant under conic degeneration:
\begin{theorem}\label{detapprox}
As $\epsilon\rightarrow 0$,
\[\log\det\Delta_{\Omega_{\epsilon}}= \frac{1}{2}(\int_N u_n(1,y)\ dy)(\log\epsilon)^2-2\log\epsilon(^{R}\zeta_{Z}(0))
+\log\det\Delta_{\Omega_{0}}+\log^{R}\det\Delta_{Z}+o(1).\] \end{theorem}
Here $u_n(1,y)$ is the coefficient of $t^0$ in the local heat asymptotics for the infinite cone $C_N$ at the point $(1,y)$; it is identically zero if $n$ is odd. Note that the coefficient of $(\log\epsilon)^2$ is $-1$ times the coefficient of $\log t$ in the heat trace asymptotics of Cheeger for manifolds with conic singularities \cite{ch2}.

To prove Theorem \ref{detapprox}, we use the representation of the zeta
function in terms of the heat trace, and then study the
precise microlocal structure of the heat
trace on $\Omega_{\epsilon}$ as a function of $t$ and $\epsilon$. 
We do not use Rowlett's analytic surgery
approach; rather, we perform a direct parametrix construction for the
heat kernel of $\Omega_{\epsilon}$, which gives us enough information
to pass to the heat trace. The key structure theorem we prove is motivated by the work of
Mazzeo and Rowlett in \cite{mr}. Let $Q$ be the quadrant $\mathbb R_+(\sqrt t)\times\mathbb R_+(\epsilon)$, and then let $Q_0$ be $[Q;\{(0,0)\}]$; that is, $Q$ with a radial blow-up at $(0,0)$ in the coordinates $(\sqrt t,\epsilon)$. The space $Q_0$ is illustrated in Figure \ref{qnaught}, and its boundary faces are labeled L, F, and R as illustrated. Let $\chi_{1}$ be a smooth radial cutoff function on $\Omega_{\epsilon}$, equal to 1 on $r\leq 15/16$ and 0 on $r\geq 17/16$ and non-increasing in $r$, and let $\chi_{2}=1-\chi_{1}$. Then we have the following:

\begin{figure}
\centering
\includegraphics[scale=0.70]{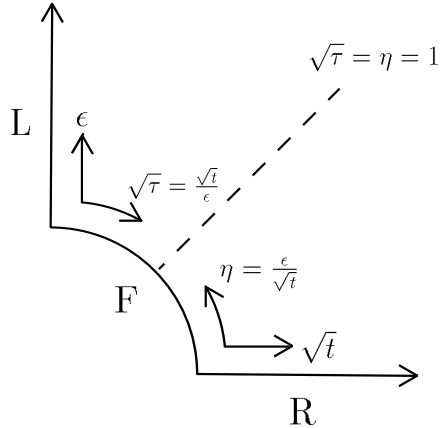}
\caption{The blown-up space $Q_0$.}\label{qnaught}
\end{figure}

\begin{theorem}\label{structure} $Tr H^{\Omega_{\epsilon}}(t)$
is polyhomogeneous conormal on $Q_{0}$, with leading orders $-n$ at L and F (in terms of $\sqrt t$) and leading order $0$ at R. Moreover, we have
\begin{equation}\label{structureeq}Tr H^{\Omega_{\epsilon}}(t)=
\int_{Z}\chi_1(\epsilon z)H^Z(\frac{t}{\epsilon^2},z,z)\ dz+\int_{\Omega_0}\chi_2(z)H^{\Omega_0}(t,z,z)\ dz+R(\epsilon,t),
\end{equation}
where $R(\epsilon,t)$ is polyhomogeneous conormal on $Q$ with infinite-order decay at $t=0$.
\end{theorem}
\begin{remark} Polyhomogeneous conormal distributions (which we sometimes abbreviate ``phg" or ``phg conormal") on manifolds with corners were originally introduced by Melrose; good introductions may be found in \cite[Chapter 5]{me}, \cite[Ch. 2A]{ma}, or \cite{gri}. They should be thought of as generalizations of smooth functions, where negative and/or fractional exponents, as well as some logarithmic terms, are allowed in the asymptotic expansions at the boundaries.
\end{remark}

Theorem \ref{structure} both describes the asymptotic structure of the heat trace and identifies all the terms in the asymptotics at L and F. On the other hand, it does not identify the leading-order term in the asymptotics of the heat trace at R, which is necessary for the proof of Theorem \ref{detapprox}. We therefore need the following heat trace convergence result as well:

\begin{theorem}\label{htconverg} As $\epsilon\rightarrow 0$, 
for each fixed positive $t$, 
\[Tr H^{\Omega_{\epsilon}}(t)\rightarrow Tr H^{\Omega_{0}}(t).\] 
Moreover, as $t\rightarrow\infty$,
$Tr H^{\Omega_{\epsilon}}(t)-1$ decays exponentially in $t$,
uniformly in $\epsilon$. 
\end{theorem}

The natural follow-up to Theorem \ref{detapprox} 
would be to obtain a similar determinant
approximation formula for the Laplacian acting on differential forms, preferably
also on sections of twisted form bundles. This would enable us to derive an
approximation formula for the analytic torsion. If we had a similar approximation
formula for the Reidemeister torsion, we could hope to apply the Cheeger-Muller
theorem on each $\Omega_{\epsilon}$ and then take a limit as $\epsilon
\rightarrow 0$ to obtain a Cheeger-Muller theorem for $\Omega_{0}$. This is
ongoing work.

\subsection{Outline of the proof}

The proofs proceed in two steps: first we assume Theorem \ref{structure} and prove the rest of the results, and then we return to prove Theorem \ref{structure}.
In section 2, we extend Theorem \ref{specconv} to cover dimensions
2, 3, and 4. We then prove a uniform lower bound on the eigenvalues of
$\Delta_{\Omega_{\epsilon}}$. Combining these two results with
Theorem \ref{structure} allows us to prove Theorem \ref{htconverg}.
Then, in section 3, we use Theorem \ref{structure} and Theorem \ref{htconverg} to prove Theorem \ref{detapprox}. We do this directly, by using the heat trace to construct the zeta
function and then analyzing the derivative at zero. This analysis is an example of how geometric microlocal structure theorems such as Theorem \ref{structure} may be powerfully combined with geometric information such as Theorems \ref{specconv} and \ref{htconverg}.

Section 4 is devoted to the proof of Theorem \ref{structure}. We construct
a parametrix for the heat equation on $\Omega_{\epsilon}$ 
by smoothly gluing together the heat kernels on $\epsilon Z$ for
$r\leq 1$ and on $\Omega_{0}$ for $r\geq 1$. The heat kernel on $\epsilon Z$
is a parabolic scaling of the heat kernel on $Z$, which
explains why we need to understand the long-time heat kernel on $Z$ as well
as the short-time heat kernel. 
We combine the results in \cite{s1} describing the structure
of the heat kernel on $Z$ with the results in \cite{mo} on the heat kernel on manifolds with conic singularities to prove Theorem \ref{structure}.
This section involves extensive use of geometric microlocal analysis, in particular Melrose's pushforward theorem, and is the most technical part of this work. For the necessary background on geometric microlocal analysis, see \cite{me,me2,gri,s,s1}. The Appendix contains the proof of one particularly technical result.

\subsection{Acknowledgements}
This work comprises the second part of my Stanford Ph.D. thesis \cite{s}. First and foremost, I would like to thank my advisor, Rafe Mazzeo, who introduced me to this problem and shared tremendous advice and support. Special thanks are also due to Colin Guillarmou for helpful comments and bug-spotting. Among the many other people who provided insight, I would like to particularly mention Xianzhe Dai, Leonid Friedlander, Andrew Hassell, and Julie Rowlett. I also wish to thank the anonymous referee for useful comments and suggestions. Finally, I am grateful to Gilles Carron and the Universit\'e de Nantes for their hospitality during fall 2010, and to the ARCS foundation for their financial support during the 2011-2012 academic year.

\section{Heat trace convergence}

In this section, we prove Theorem \ref{htconverg}. Our strategy is to combine the spectral convergence of \cite{ma2},
\cite{row}, and \cite{at} with a uniform lower bound on the eigenvalues.
The uniform lower bound will allow us to control the remainders in
the infinite sum
\[Tr H^{\Omega_{\epsilon}}(t)=\sum_{k=0}^{\infty}e^{-\lambda_{\epsilon,k}t}.\]

\subsection{Spectral convergence}

As discussed in \cite{row2}, the spectral convergence result of Theorem \ref{specconv} has only been demonstrated when $n\geq 5$. We must therefore extend Theorem \ref{specconv} to the cases where $2\leq n\leq 4$.  Our approach is the same as that in \cite{ma2} and \cite{row}. The key is to prove that if we have a sequence of eigenvalues $\lambda_{\epsilon}$ of 
$\Delta_{\Omega_{\epsilon}}$ which converge to $\bar{\lambda}$, then
$\bar{\lambda}$ is an eigenvalue of the Friedrichs extension of
$\Delta_{\Omega_{0}}$. This will show that every accumulation point
of $\sigma(\Delta_{\Omega{\epsilon}})$ is in $\sigma(\Delta_{\Omega_{0}})$.

Assume that for each $\epsilon>0$, $f_{\epsilon}$ is a nonzero eigenfunction
of $\Delta_{\epsilon}$ with eigenvalue $\lambda_{\epsilon}$, normalized so
that $||f_{\epsilon}||_{\infty}=1$. Over any compact set in $\Omega_{0}$
away from the conic tip, 
the Arzela-Ascoli theorem implies that $f_{\epsilon}$ converges
uniformly to a limiting function $\bar f$. By elliptic estimates, the convergence
is actually in $C^{\infty}$, so the limit $\bar f$ is smooth and satisfies 
$\Delta_{\Omega_{0}}\bar f=\bar\lambda\bar f$. 
Note that since each $f_{j}$ is bounded by $1$, we also
have $||\bar f||_{\infty}\leq 1$, which automatically implies that
$\bar f$ is in the Friedrichs domain of $\Delta_{\Omega_{0}}$. 
In order to show that $\bar{\lambda}$ is an eigenvalue of
$\Delta_{\Omega_{0}}$, we must show that $\bar f$ is nonzero.

Following \cite{ma2} and \cite{row}, we define a weight function $R_{\epsilon}$
on $\Omega_{\epsilon}$ for each $\epsilon$. We set $R_{\epsilon}=\epsilon$
for $r\leq\epsilon$, $R_{\epsilon}=r$ for 
$\epsilon\leq r\leq 1$, and $R_{\epsilon}=1$ outside
$r=1$. Let $\delta$ be a small positive number which we will choose later
in the argument. For each $\epsilon$, let $p_{\epsilon}$ be the point
at which the supremum of $|f_{\epsilon}R_{\epsilon}^{\delta}|$ is attained.
Re-normalize the $f_{\epsilon}$ by writing $\varphi_{\epsilon}=
f_{\epsilon}|R_{\epsilon}^{\delta}f_{\epsilon}(p_{\epsilon})|^{-1}$,
so that the supremum of $|R_{\epsilon}^{\delta}\varphi_{\epsilon}|$ is 1 
and it is achieved at $p_{\epsilon}$. 
Note that the $\varphi_{\epsilon}$ are still eigenfunctions
of $\Delta_{\Omega_{\epsilon}}$.

We now analyze the behavior of $R_{\epsilon}(p_{\epsilon})$ as $\epsilon$
approaches zero. Passing to a subsequence, we may assume that we are in
one of three cases.

\medskip

\textbf{Case 1:} Suppose that 
$R_{\epsilon}(p_{\epsilon})$ approaches $R>0$ as $\epsilon$
goes to zero. In this case, some subsequence of $p_{\epsilon}$ approaches some
$\bar p$ in $\Omega_{0}$ away from the conic tip. Arguing as before, the
$\varphi_{\epsilon}$ converge to some function $\bar\varphi$ on compact subsets
away from the conic tip,
and since the renormalizing factor $R_{\epsilon}^{\delta}(p_{\epsilon})$ 
converges to $R^{\delta}$, we have $\bar\varphi=R^{-\delta}\bar f$. Since 
$|R_{\epsilon}^{\delta}\varphi_{\epsilon}(p_{\epsilon})|=1$ for all $\epsilon$,
we must have $|\bar\varphi(p)|=R^{-\delta}$. Therefore $|\bar f(p)|=1$, 
so $\bar f$ is nonzero; we already
showed that $\bar f$ is in the Friedrichs domain, so we conclude that
$\bar\lambda\in\sigma(\Delta_{\Omega_{0}})$.

\medskip

\textbf{Case 2:} Suppose that there exists a constant $C$ so that
$R_{\epsilon}(p_{\epsilon})\leq C\epsilon$ as $\epsilon$ goes to zero. We
will derive a contradiction. Restricting $\varphi_{\epsilon}$ to the
region where $r\leq 1$, we consider $\varphi_{\epsilon}$ as an eigenfunction
of $\Delta_{\epsilon Z}$ with eigenvalue $\lambda_{\epsilon}$. Rescaling
by $\epsilon$, we view $\varphi_{\epsilon}$ as a function on 
$Z$; for $|z|\leq 1/\epsilon$, it is an
eigenfunction of $\Delta_{Z}$ with eigenvalue $\epsilon^{2}\lambda_{\epsilon}$.

For each $\epsilon$, let $\psi_{\epsilon}=R_{\epsilon}(p_{\epsilon})^{\delta}
\varphi_{\epsilon}$; it is also an
eigenfunction of $\Delta_{Z}$ with the same eigenvalue, and we have
$|\psi_{\epsilon}(p_{\epsilon})|=1$. Moreover, $p_{\epsilon}$ maximizes
$R_{\epsilon}^{\delta}|\varphi_{\epsilon}|$, so it also maximizes
$R_{\epsilon}^{\delta}|\psi_{\epsilon}|$. Since $R_{\epsilon}$
rescales to a function on $Z$ which is $\epsilon |z|$ 
outside $|z| =1$, we have that for $|z| \geq 1$,
\[(\epsilon|z|)^{\delta}|\psi_{\epsilon}(z)|
\leq R_{\epsilon}(p_{\epsilon})^{\delta}\leq C^{\delta}\epsilon^{\delta}.\]
Therefore, for $|z|\geq 1$, 
we have $|\psi_{\epsilon}(z)|\leq C^{\delta}|z|^{-\delta}$.

As before, we take a limit of the $\psi_{\epsilon}$
(uniformly on compact subsets of $Z$) to obtain a limit $\bar\psi$.
Since $\epsilon^{2}\lambda_{\epsilon}$ goes to zero, $\bar\psi$ is a harmonic
function on $Z$.
The rescaled points $p_{\epsilon}$ lie in a bounded subset of $Z$ and hence
converge to some $\bar p\in Z$, and so we must have 
$\bar\psi(\bar p)=1$. Since each $|\psi_{\epsilon}(z)|\leq C^{\delta}|z|^{-\delta}$
for $|z|\geq 1$, the limit $\bar\psi$ must satisfy the same bound. However,
since $\delta>0$, $\bar\psi$ is a nonzero harmonic function on $Z$ which
approaches zero at infinity. This is impossible by the maximum principle,
and we have a contradiction.

\medskip

\textbf{Case 3:} Finally, suppose that $R_{\epsilon}(p_{\epsilon})$ approaches
zero but that $\epsilon^{-1}R_{\epsilon}(p_{\epsilon})$ goes to infinity.
Consider the restriction of $\varphi_{\epsilon}$ to the region in $\Omega_{\epsilon}$
between $\epsilon$ and 1, where the metric is exactly conic and $R_{\epsilon}$
is precisely $r$. We again rescale, but this time by $R_{\epsilon}(p_{\epsilon})$
instead of $\epsilon$.
The result is that we view $\varphi_{\epsilon}$
 as a function on the region of the infinite cone $C_{N}$ between
$s=\epsilon/(R_{\epsilon}(p_{\epsilon}))$ and $s=1/(R_{\epsilon}(p_{\epsilon}))$,
where $s$ is the radial variable on the cone. Note that $p_{\epsilon}\in\{s=1\}$.

As in Case 2, we again let $\psi_{\epsilon}=R_{\epsilon}(p_{\epsilon})^{\delta}
\varphi_{\epsilon}$.
This rescaled function $\psi_{\epsilon}$ is an eigenfunction of $\Delta_{C_{N}}$
with eigenvalue $R_{\epsilon}(p_{\epsilon})^{2}\lambda_{\epsilon}$, and
$|\psi_{\epsilon}|$ attains its maximum of 1 at $p_{\epsilon}$.
The function $R_{\epsilon}$ lifts to $R_{\epsilon}(p_{\epsilon})s$, so
we have that for any point $(s,\theta)\in C_{N}$ with $s$ between
$\epsilon/(R_{\epsilon}(p_{\epsilon}))$ and $1/(R_{\epsilon}(p_{\epsilon}))$,
\[(R_{\epsilon}(p_{\epsilon})s)^{\delta}|\psi_{\epsilon}(s,\theta)|
\leq R_{\epsilon}(p_{\epsilon})^{\delta}.\]
Hence $|\psi_{\epsilon}(s,\theta)|\leq s^{-\delta}$.

As $\epsilon$ goes to zero, the lower bound for $s$ goes to zero
and the upper bound for $s$ goes to infinity. Therefore, as in the
other cases, the functions $\psi_{\epsilon}$ converge to a limit
$\bar\psi$ on $C_{N}$. The rescaled points $p_{\epsilon}$ all have $s=1$ and
hence have a subsequence converging to a point $\bar p$. The eigenvalues
$R_{\epsilon}(p_{\epsilon})^{2}\lambda_{\epsilon}$ converge to zero,
so $\bar\psi$ is a harmonic function on $C_{N}$; moreover, $\bar\psi(p)=1$, so
it is nonzero. Finally, $\bar\psi(s,\theta)\leq s^{-\delta}$ for all $s$.
If $\delta$ is not an indicial root of the conic Laplacian, this is impossible
(see \cite{ma2}).
Since the indicial roots are discrete, we may pick $\delta$ which is not
an indicial root. This contradiction completes the proof that $\bar\lambda$ is
in $\sigma(\Delta_{\Omega_{0}})$.

\medskip

In order to complete the proof of spectral convergence,
we must also show that every point of $\sigma(\Delta_{\Omega_{0}})$ is an
accumulation point of $\sigma(\Delta_{\Omega_{\epsilon}})$, and that the
multiplicities are correct. However, the proof of these assertions
in \cite{row} does not depend on the assumption 
$n\geq 5$. Therefore, the remainder
of the proof proceeds as in \cite{row}, and we will not repeat it here.
We have now established Theorem \ref{specconv} for all $n\geq 2$.

\subsection{Lower bounds for eigenvalues}

We now prove uniform lower bounds on the eigenvalues of $\Delta_{\Omega_{\epsilon}}$. This sort of lower bound is generally difficult to prove using classical methods. However, as we will see, it is a relatively straightforward consequence of Theorem \ref{structure}. The key step is the following upper bound on the heat trace:
\begin{lemma}\label{auxiliaryone} For any fixed $T>0$, there is a constant $C$ so that
for all $t\leq T$ and all $\epsilon<1/2$ (including $\epsilon=0$),
\[Tr H^{\Omega_{\epsilon}}(t)\leq Ct^{-n/2}.\]
\end{lemma}
\begin{proof} From Theorem \ref{structure}, $Tr H^{\Omega_\epsilon}(t)$ is polyhomogeneous conormal on $Q_0$, with leading orders $-n$ at L and F (in terms of $\sqrt t$) and 0 at R. The function $t^{n/2}$ is polyhomogeneous on $Q$ and hence on $Q_{0}$,
with leading orders $n$ at L and F and 0 at R. Therefore, $t^{n/2}Tr H^{\Omega_{\epsilon}}(t)$
is polyhomogeneous on $Q_{0}$ with leading order 0 at each boundary face. In particular, $t^{n/2}Tr H^{\Omega_\epsilon}(t)$ is continuous up to the boundary of $Q_0$. If we
restrict to $t\leq T$ for any fixed $T$ and to $\epsilon<1/2$, then $t^{n/2}(Tr H^{\Omega_{\epsilon}}(t))$ is uniformly bounded, proving the lemma.
\end{proof}

Lemma \ref{auxiliaryone} immediately implies a uniform lower bound for the
eigenvalues:
\begin{lemma}\label{eigenbound} There is an $N_0\in\mathbb N$ 
and a $C'>0$ such that for all $k\geq N_0$ and all $\epsilon<1/2$
(including $\epsilon=0$),
\[\lambda_{\epsilon,k}\geq C'k^{2/n}.\]
\end{lemma}
\begin{proof} Let $N_{\epsilon}(\lambda)$ be the number of eigenvalues of
$\Delta_{\Omega_{\epsilon}}$ less than or equal to $\lambda$. Then apply
Lemma \ref{auxiliaryone} with $T=n/2$ to obtain that for all $t\leq n/2$,
\begin{equation}N_{\epsilon}(\lambda)e^{-\lambda t}\leq\sum_{\lambda_{\epsilon,j}\leq\lambda}
e^{-\lambda_{\epsilon,j}t}\leq Tr H^{\Omega_{\epsilon}}(t)\leq Ct^{-n/2}.
\end{equation}
Therefore, for all $t\leq n/2$,
\begin{equation}\label{countingbound} 
N_{\epsilon}(\lambda)\leq Ct^{-n/2}e^{\lambda t}.
\end{equation}
The right side of (\ref{countingbound}) is minimized when $t=(n/2)/\lambda$,
which is bounded by $n/2$ when $\lambda\geq 1$. So when $\lambda\geq 1$,
\begin{equation}N_{\epsilon}(\lambda)\leq \tilde C\lambda^{n/2}.
\end{equation}
Applying this to $\lambda=\lambda_{\epsilon,k}$, we see that for
$\lambda_{\epsilon,k}\geq 1$,
\[k\leq\tilde C\lambda_{\epsilon,k}^{n/2},\]
and hence
\[\lambda_{\epsilon,k}\geq C'k^{2/n}.\]
Finally, note that $N_{\epsilon}(1)\leq\tilde C$, so picking any $N_0\geq\tilde C$
ensures that if $k\geq N$, $\lambda_{\epsilon,k}\geq 1$. This completes
the proof of Lemma \ref{eigenbound}.
\end{proof}

This section illustrates the power of microlocal structure theorems such as Theorem \ref{structure}; it is possible to prove Lemma \ref{eigenbound} via a classical argument involving isoperimetric inequalities, but that method is much more involved.

\subsection{Heat trace convergence}
We now prove Theorem \ref{htconverg}, starting with the exponential decay claim. By Theorem \ref{specconv}, $\lambda_{\epsilon,1}$ converges to $\lambda_{0,1}>0$, and hence is bounded below uniformly by
$\lambda_{0,1}/2$ for small $\epsilon$; the same is of course true
for $\lambda_{\epsilon,k}$ for all $k$. Combining this observation with Lemma \ref{eigenbound}, we have
\begin{equation}\label{onemore}
Tr H^{\Omega_{\epsilon}}(t)\leq 1+(N_0-1)e^{-(\lambda_{0,1}/2)t}+
\sum_{k=N_0}^{\infty}e^{-tC'k^{2/n}}.\end{equation}
Since the infinite sum in (\ref{onemore}) is convergent for any $t$ and 
exponentially decaying in $t$, 
subtracting 1 from both sides immediately proves the second statement
in Theorem \ref{htconverg}.

To prove the heat trace convergence part of Theorem \ref{htconverg}, 
fix $\delta>0$ and $t>0$.
For any $M\geq N_0$, we write
\[|Tr H^{\Omega_{\epsilon}}(t)-Tr H^{\Omega_{0}}(t)|\leq
\sum_{k=1}^{M}|e^{-\lambda_{\epsilon,k}t}-e^{-\lambda_{\epsilon,0}t}|
+\sum_{k=M+1}^{\infty}e^{-\lambda_{\epsilon,k}t}+\sum_{k=M+1}^{\infty}
e^{-\lambda_{0,k}t}.\] 
Lemma \ref{eigenbound} implies that the second and third terms are both bounded by
\[\sum_{k=M+1}^{\infty}e^{-tC'k^{2/n}},\]
which goes to zero as $M$ goes to infinity. So we may fix an $M$ such that
both of these terms are bounded by $\delta/3$ for any $\epsilon<1/2$. 
Finally, by spectral convergence, for sufficiently small $\epsilon$,
the first term is also bounded by $\delta/3$. We conclude that for sufficiently
small $\epsilon$,
\[|Tr H^{\Omega_{\epsilon}}(t)-Tr H^{\Omega_{0}}(t)|<\delta.\]
Since $d>0$ was arbitrary, this completes the proof of Theorem \ref{htconverg}.


\section {The determinant formula}

The proof of Theorem \ref{detapprox} involves direct analysis of the zeta function on $\Omega_\epsilon$, in particular the standard heat trace formula
\begin{equation}\label{zetafunction}
\zeta_{\Omega_{\epsilon}}(s)=\frac{1}{\Gamma(s)}\int_{0}^{\infty}
(Tr H^{\Omega_{\epsilon}}(t)-1) t^{s-1}\ ds,\end{equation}
where equality is in the sense of meromorphic continuation from the
region where $\Re s>n/2$. We prove Theorem \ref{detapprox} by decomposing
the heat trace into several pieces and analyzing (\ref{zetafunction})
for each piece, using Theorems \ref{structure} and \ref{htconverg}.

\subsection{Long-time contribution}
We begin by studying the long-time piece of the zeta function:
\begin{equation}\label{zetafunctionlong}
\zeta_{\Omega_{\epsilon},L}(s)=\frac{1}{\Gamma(s)}\int_{1}^{\infty}
(Tr H^{\Omega_{\epsilon}}(t)-1) t^{s-1}\ ds.\end{equation}

The integrand decays exponentially in $t$ as $t\rightarrow\infty$
for any $s\in\mathbb C$, so we can differentiate under the integral sign.
The function $\frac{1}{\Gamma(s)}$ is holomorphic at $s=0$ with leading order $s$,
so differentiation gives:
\begin{equation}\label{zetaprimelong}
\zeta'_{\Omega_{\epsilon},L}(0)=\int_{1}^{\infty}(Tr H^{\Omega_{\epsilon}}(t)-
1)
t^{-1}\ dt.\end{equation} 
As an aside, note that in fact $\zeta_{\Omega_{\epsilon,L}}(0)=0$;
the integral in (\ref{zetafunctionlong}) is holomorphic on $\mathbb C$,
and hence multiplication by the inverse gamma function gives a holomorphic
function which is 0 at $s=0$.

By Theorem \ref{htconverg}, there exist $\mu>0$ and $C<\infty$
independent of $\epsilon$ such that:
\[|(Tr H^{\Omega_{\epsilon}}(t)-1)|\leq Ce^{-\mu t}.\]
Moreover, by the same theorem,
\[\lim_{\epsilon\rightarrow 0} Tr H^{\Omega_{\epsilon}}(t)-1=
Tr H^{\Omega_{0}}(t)-1.\]
We may therefore apply the dominated convergence theorem to
(\ref{zetaprimelong}) to conclude that
\begin{equation}\label{longtimelimit}
\lim_{\epsilon\rightarrow 0}\zeta'_{\Omega_{\epsilon},L}(0)= \zeta'_{\Omega_{0},L}(0).\end{equation}
In fact, we expect that there should
be a polyhomogeneous expansion for $\zeta'_{\Omega_{\epsilon},L}(0)$ as
$\epsilon$ goes to zero, but it is difficult to analyze the long-time behavior
of $\epsilon$-derivatives of $Tr H^{\Omega_{\epsilon}}(t)$.

\subsection{Projection term}
Now we consider the short-time zeta function
\begin{equation}
\zeta_{\Omega_{\epsilon},S}(s)=\frac{1}{\Gamma(s)}\int_{0}^{1}
(Tr H^{\Omega_{\epsilon}}(t)-1) t^{s-1}\ ds.\end{equation}
By direct computation, we may write
\begin{equation}\label{stzetacomplete}
\zeta_{\Omega_{\epsilon},S}(s)=\frac{1}{\Gamma(s)}\int_{0}^{1}
Tr H^{\Omega_{\epsilon}}(t)t^{s-1}\ ds -\frac{1}{s\Gamma(s)}.
\end{equation}
The second term in (\ref{stzetacomplete}) is the contribution
from the projection off the constants. 

\subsection{Leading order term at R}\label{leadingorderterm}
Now that we have analyzed the long-time and projection contributions, we 
consider the integral
\begin{equation}\label{shorttimezeta}
\frac{1}{\Gamma(s)}\int_{0}^{1} Tr H^{\Omega_{\epsilon}}(t)\ t^{s-1}\ dt.
\end{equation}.
To best take advantage of our knowledge about the leading-order terms of
$Tr H^{\Omega_{\epsilon}}(t)$ on $Q_{0}$, we break off the leading-order
term at R in a neighborhood of R. In particular, fix $b\in(0,\infty)$,
and break the analysis of (\ref{shorttimezeta}) into two regions: 
$\sqrt t<b\epsilon$ and $\sqrt t>b\epsilon$. We work in local coordinates in each region; let $\tau=t/\epsilon^2$ and $\eta=\epsilon/\sqrt t$, as inspired by \cite{mr} and illustrated in Figure \ref{qnaught}.
When $\sqrt t<b\epsilon$, we use the coordinates $(\sqrt\tau,\epsilon)$;
when $\sqrt t>b\epsilon$, we instead use $(\eta,\sqrt t)$.

In the region $\{\sqrt t>b\epsilon\}$, we break off the leading-order term at
R, which by Theorem \ref{htconverg} is 
$Tr H^{\Omega_{0}}(t)\chi\{\sqrt t\geq b\epsilon\}$, and consider it separately.
Its contribution to (\ref{shorttimezeta}) is precisely
\begin{equation}\label{equationeleven}
\frac{1}{\Gamma(s)}\int_{b^{2}\epsilon^{2}}^{1}
Tr H^{\Omega_{0}}(t)t^{s-1}\ dt=
(\zeta_{\Omega_{0},S}(s)+\frac{1}{s\Gamma(s)})
-\frac{1}{\Gamma(s)}\int_{0}^{b^{2}\epsilon^{2}}Tr H^{\Omega_{0}}(t)t^{s-1}\ dt.
\end{equation}
Note that (\ref{equationeleven}) 
is initially valid for large $s$; 
however, since both sides have meromorphic continuations,
it is valid for all $s$.

As proven by Cheeger \cite{ch2} and generalized by Bruning and Seeley \cite{brs}, $Tr H^{\Omega_{0}}(t)$ has a short-time heat expansion given by
\[Tr H^{\Omega_{0}}(t)=\sum_{k=0}^{n}a_{k}t^{(k-n)/2}+
K\log t+\mathcal R(t),\]
where $\mathcal R(t)$ is $\mathcal O(t^{\alpha})$ for some $\alpha>0$. We plug this expansion into the right-hand side of (\ref{equationeleven}).
All the integrals except for the error term may be evaluated
directly, and yield:
\[-\frac{1}{\Gamma(s)}\sum_{k=0}^{n}
a_{k}\frac{(b\epsilon)^{k-n+2s}}{(k-n)/2+s}
+\frac{Kb^{2s}\epsilon^{2s}}{s\Gamma(s)}(\frac{1}{s}-(2\log b+2\log\epsilon)).\]
This is a meromorphic function on $\mathbb C$, 
with a simple pole at $0$ if $K\neq 0$.

As for the error term, it is of order $t^{\alpha}$ for $\alpha>0$,
so its contribution to (\ref{equationeleven}) converges and is holomorphic
in a neighborhood
of $s=0$. It is bounded by
\[\frac{1}{\Gamma(s)}\int_{0}^{b^{2}\epsilon^{2}}Ct^{\alpha+s-1}\ dt
=\frac{C}{\Gamma(s)(\alpha+s)}(b^{2}\epsilon^{2})^{\alpha+s}.\]
In the neighborhood $|s|<\alpha/2$, this is a holomorphic function for 
each $\epsilon$ which decays in $\epsilon$. Thus 
the contribution of this error term to the zeta function is $o(1)$ in $\epsilon$. To show that the same is true for the determinant, we take an $s$-derivative of the $\mathcal R(t)$ portion of (\ref{equationeleven}) and obtain:
\begin{equation}\label{sderiv}-\int_0^{b^2\epsilon^2}\mathcal R(t)t^{s-1}\ dt-\frac{1}{\Gamma(s)}\int_0^{b^2\epsilon^2}\mathcal R(t)t^{s-1}\log t\ dt.\end{equation} Essentially the same analysis, using the fact that $\mathcal R(t)<Ct^{\alpha}$, shows that (\ref{sderiv}) is $o(1)$ in $\epsilon$.\\

Combining the contributions to (\ref{zetafunction}) of the long-time zeta function, the projection off the constants, and the leading-order term in $\sqrt t>b\epsilon$, we have:
\begin{equation}\label{leadingproj}
\zeta_{\Omega_{0}}(s)-\frac{1}{\Gamma(s)}\sum_{k=0}^{n}
a_{k}\frac{(b\epsilon)^{k-n+2s}}{(k-n)/2+s}
+\frac{Kb^{2s}\epsilon^{2s}}{s\Gamma(s)}(\frac{1}{s}-(2\log b+2\log\epsilon)),
\end{equation}
plus a term whose contribution to the
determinant is $o(1)$ as $\epsilon$ goes to zero. Notice that the factors of $(s\Gamma(s))^{-1}$ from the projections off the constants in the short-time and long-time contributions cancel.

\subsection{The $\Omega_{0}$ term}
Now we consider the portion of (\ref{shorttimezeta}) which remains after
subtracting the leading-order term at R. Recall the definitions of $\chi_1(z)$ and $\chi_2(z)$ from the discussion in the introduction preceding Theorem \ref{structure}. Write
\begin{equation}\label{breakup} Tr H^{\Omega_{\epsilon}}(t)=
\int_{\epsilon Z}\chi_{1}(z)H^{\epsilon z}(t,z,z)\ dz+\int_{\Omega_{0}}
\chi_{2}(z)H^{\Omega_{0}}(t,z,z)\ dz+R(\epsilon, t).\end{equation}
We analyze these three terms' contributions to
(\ref{shorttimezeta}) separately; however, we must subtract the leading order term at R in the region $\sqrt t>b\epsilon$. By Theorem \ref{structure} and its proof, each of the three terms is phg conormal on $Q_0$ with leading order 0 at R, so we may separately subtract the leading order term at R from each term in (\ref{breakup}).

First, we consider the $\Omega_0$ term. Its contribution to (\ref{shorttimezeta}) is
\[\frac{1}{\Gamma(s)}\int_{0}^{1}(\int_{\Omega_{0}}\chi_{2}(z)H^{\Omega_{0}}
(t,z,z)\ dz)
t^{s-1}\ dt;\] 
however, we have to subtract the leading-order terms at R in $\sqrt t>b\epsilon$. The $\Omega_{0}$ term is
independent of $\epsilon$, so the leading-order term at R is the whole integral.
Therefore, we only integrate from $0$ to $b^{2}\epsilon^{2}$. The contribution
to (\ref{shorttimezeta}) is
\[\frac{1}{\Gamma(s)}\int_{0}^{b^{2}\epsilon^{2}}
(\int_{\Omega_{0}}\chi_{2}(z)H^{\Omega_{0}}(t,z,z)\ dz)t^{s-1}\ dt.\]
This analysis is similar to the previous section. Note that the short-time
heat expansion for $\Omega_{0}$ is local, and $\chi_{2}(z)$ localizes
away from the cone point, so there is no logarithmic term. Therefore,
there exist $\tilde a_{k}$ such that
\[\int_{\Omega_{0}}\chi_{2}(z)H^{\Omega_{0}}(t,z,z)\ dz=
\sum_{k=0}^{n}\tilde a_{k}t^{(k-n)/2}+\mathcal O(t^{\alpha}).\]
As in the discussion of the leading-order term at R, we obtain a contribution
to (\ref{shorttimezeta}) of
\begin{equation}\label{omeganoughtcontrib}
\frac{1}{\Gamma(s)}\sum_{k=0}^{n}
\tilde a_{k}\frac{(b\epsilon)^{k-n+2s}}{(k-n)/2+s},\end{equation}
modulo an error term whose contribution to the determinant is $o(1)$ in
$\epsilon$.

\subsection{The error term}\label{theerrorterm}
Now we consider the term coming from $R(\epsilon,t)$. In the region
$\sqrt t<b\epsilon$, this decays to infinite order in both
$\tau$ and $\epsilon$; in particular, it is less than $C\tau\epsilon^{2}=Ct$
for some fixed $C$. As before, any term of order $\mathcal O(t^{\alpha})$
for $\alpha>0$ contributes only $o(1)$ to the determinant when integrated
over the region $\sqrt t<b\epsilon$, so the contribution from this region
is $o(1)$.

On the other hand, where $\sqrt t>b\epsilon$, $R(\epsilon,t)$
is a function of $(\eta,\sqrt t)$. It decays to infinite order in $\sqrt t$,
and since we have subtracted the leading order term at R it decays
to some positive order $\delta$ in $\eta$. In particular,
$R(\epsilon,t)$ is bounded in the region $\sqrt t>b\epsilon$ by
$C\sqrt t^{3\delta}\eta^{\delta}=Ct^{\delta}\epsilon^{\delta}$.
The contribution to (\ref{shorttimezeta}) is an integral
between $t=b^{2}\epsilon^{2}$ and $t=1$, so it is convergent
and holomorphic for all $s$ (we do not integrate down to $s=0$).
We may then differentiate under the integral sign and obtain that
its contribution to $\zeta_{\Omega_{\epsilon},S}'(0)$ is
\[\int_{b^{2}\epsilon^{2}}^{1}R(\epsilon,t)t^{-1}\ dt,\]
which is bounded by
\[C\epsilon^{\delta}\int_{0}^{1}t^{\delta-1}\ dt=\frac{C}{\delta}\epsilon^{\delta}.\]
Therefore, the remainder term $R(\epsilon,t)$
contributes only $o(1)$ to the determinant
as $\epsilon$ approaches zero after we subtract the leading order at R.

\subsection{The $\epsilon Z$ term}\label{epsilonzee}

It remains only to analyze the contribution to (\ref{shorttimezeta})
of the first term in (\ref{breakup}).
When we plug this term into (\ref{shorttimezeta}), we get
\[\frac{1}{\Gamma(s)}\int_{0}^{1}(\int_{\epsilon Z}\chi_{1}(z)
H^{\epsilon Z}(t,z,z)\ dz)t^{s-1}\ dt,\]
which may be rewritten, by scaling and a change of variables, as:
\begin{equation}\label{anotherthingtolabel}\frac{2\epsilon^{2s}}{\Gamma(s)}\int_{0}^{1/\epsilon}(\int_{Z}\chi_{1}(\epsilon z)
H^{Z}(\tau,z,z)\ dz)\sqrt\tau^{2s-1}\ d\sqrt\tau. \end{equation}

We now recognize the term
\begin{equation}\label{rememberme}\int_{Z}\chi_{1}(\epsilon z)
H^{Z}(\tau,z,z)\ dz\end{equation} from the discussion in \cite{s1} of the renormalized heat trace on asymptotically conic manifolds. Recall from \cite{s1} that the renormalized heat trace $^RTr H^Z(\tau)$ is defined to be the constant term in the expansion as $\epsilon$ goes to zero of
\[\int_{|z|\leq 1/\epsilon}H^Z(\tau,z,z).\]
In \cite{s1}, we prove that (\ref{rememberme}) is polyhomogeneous conormal in $(\tau,\epsilon)$ for small $\tau$ and on $X_b^2(\tau^{-1/2},\epsilon)$ for large $\tau$.\footnote{The definition of $X_b^2(\tau^{-1/2},\epsilon)$ may be found in \cite{ms}; it is just $\mathbb R_+(\tau^{-1/2})\times\mathbb R_+(\epsilon)$ with a radial blow-up at $\tau^{-1/2}=\epsilon=0$. As always, see \cite{me,me2,gri} for background on blow-ups and polyhomogeneous conormal functions.} Moreover, since $Z$ is exactly conic near infinity, \cite[Lemma 19]{s1} gives an explicit expansion for (\ref{rememberme}) as $\epsilon$ goes to zero:
\[
\int_{Z}\chi_{1}(\epsilon z)H^{Z}(\tau,z,z)\ dz=
\sum_{k=0}^{n-1}D_{k}\tau^{(k-n)/2}\epsilon^{k-n}-L\log\epsilon\]
\begin{equation}\label{smoothexpredux}+(^{R}Tr H^{Z}(\tau)-Ll_{log})
+\tilde R(\epsilon,\tau),\end{equation}
where $\tilde R(\epsilon,\tau)$ vanishes as $\epsilon\rightarrow 0$ for each $\tau$. Moreover, if we let $u_n(z)$ be the coefficients of the heat asymptotics on the infinite cone $C_N$, then
 \[L=\int_{N}u_{n}(1,y)\ dy,\ \ D_{k}=\frac{l_{k}}{k-n}\int_{N}u_{k}(1,y)\ dy,\]\[
l_{k}=-\int_{1/2}^{2}\chi_{1}'(r)r^{k-n}\ dr, l_{log}=-\int_{1/2}^{2}\chi_{1}'(r)\log r\ dr.\]
Note that by the phg conormality of (\ref{rememberme}) and the expansion (\ref{smoothexpredux}), the remainder term $\tilde R(\epsilon,\tau)$ has the same phg conormality as (\ref{rememberme}).

We will plug (\ref{smoothexpredux}) into (\ref{shorttimezeta}).
Since we still need to subtract off the leading order terms in 
$\sqrt\tau>b$, we break up the integral at $\sqrt\tau=b$.

\subsubsection{$\tau<b^2$:}

For $\tau<b^{2}$, $(\epsilon,\tau)$ are good coordinates, 
and we rewrite the part of the integral (\ref{anotherthingtolabel}) where $\tau<b^2$ as an
integral in $\sqrt\tau$:
\begin{equation}\label{shorttimezetatau}
\frac{2\epsilon^{2s}}{\Gamma(s)}\int_{0}^{b}Tr H^{\Omega_{\epsilon}}(\tau\epsilon^2)
\sqrt\tau^{2s-1}d\sqrt\tau.\end{equation}
Plugging the expansion (\ref{smoothexpredux}) into (\ref{shorttimezetatau}) and integrating gives
\[\sum_{k=0}^{n-1}\frac{2D_{k}\epsilon^{2s+k-n}b^{2s+k-n}}{(2s+k-n)\Gamma(s)}-
\frac{L\epsilon^{2s}b^{2s}\log\epsilon}{s\Gamma(s)}
-\frac{Ll_{log}\epsilon^{2s}b^{2s}}{s\Gamma(s)}\]\[
+\frac{2\epsilon^{2s}}{\Gamma(s)}\int_{0}^{b}\ ^{R}Tr H^{Z}(\tau)\sqrt\tau^{2s-1}\ d\sqrt\tau
+\hat R(\epsilon,s),\]
where \begin{equation}\label{remainone}\hat R(\epsilon,s)=
\frac{2\epsilon^{2s}}{\Gamma(s)}\int_{0}^{b}\tilde R(\epsilon,\tau)\sqrt\tau^{2s-1}\
d\sqrt\tau.\end{equation}

Note that $\tilde R(\epsilon,\tau)$ is phg conormal
in $(\epsilon,\tau)$ for small $\tau$ and the leading order
at $\epsilon=0$ is positive, so the coefficients of powers of
$\tau$ all decay to a positive order in $\epsilon$; let $\delta$ be less than that positive power. Then there exist $l$, $z_k$, and $p_k$ for each $k$ between $0$ and $l$ such that
\[\tilde R(\epsilon,\tau)=\sum_{k=0}^{l}b_{k}(\epsilon)\tau^{z_{k}}(\log\tau)^{p_{k}}
+\mathcal O(\epsilon^{\delta}\tau^{\delta/2}),\]
where $b_{k}(\epsilon)$ are $\mathcal O(\epsilon^{\delta})$ as well.
Plugging this expansion into (\ref{remainone}), the 
$\mathcal O(\epsilon^{\delta}\tau^{\delta/2})$ term is bounded by $C\epsilon^{\delta}\tau^{\delta/2}=t^{\delta/2}$ for $\sqrt\tau<b$, and hence its contribution to the zeta function and determinant is $o(1)$, as before. The rest of the expansion gives
\[\frac{\epsilon^{2s}}{\Gamma(s)}\sum_{k=0}^{l}b_{k}(\epsilon)
\int_{0}^{b}\tau^{z_{k}+2s-1}(\log\tau)^{p_{k}}\ d\sqrt\tau.\]
Each term, in a neighborhood of $s=0$, is $\epsilon^{2s+\delta}$ times an
explicit meromorphic function of $s$ with coefficient depending on $b$.
Therefore, each of the coefficients in the Laurent series for (\ref{remainone}) around $s=0$ is $\mathcal O(\epsilon^{\delta})$, and so the contribution of this remainder term to the zeta function and determinant is $o(1)$ (in fact, $\mathcal O(\epsilon^{\delta})$).

\subsubsection{$\tau>b^2$:}

In this region, we use the coordinates $(\eta,\sqrt t)$, and the expansion (\ref{smoothexpredux}) becomes:
\begin{equation}\label{smoothexpnearr}
\sum_{k=0}^{n-1}D_{k}(\sqrt t)^{k-n}-L\log(\sqrt t)-L\log\eta
+^{R}Tr H^{Z}(\eta^{-2})-Ll_{log}+\tilde R.\end{equation}
We need to subtract the leading order term at R, which in these coordinates is the face $\eta=0$. We do this first for the non-remainder terms and then for $\tilde R$.

The expansion (\ref{smoothexpnearr}) is phg conormal on $Q_0$, hence in $(\eta,\sqrt t)$ in this region, and it has leading order 0 at R. The non-remainder terms are phg conormal as a consequence of the phg conormality of the renormalized heat trace (proven in \cite{s1}). Moreover, they are precisely the finite and divergent parts of (\ref{rememberme}) at F, and hence have leading order 0 at R; this implies that $^RTr H^Z(\eta^{-2})$ has leading-order behavior of $L\log\eta$ at $\eta=0$. If we let $f_{\infty}$ be the constant term in the expansion at $\eta=0$ of $^{R}Tr H^{Z}(\eta^{-2})$, we see that the terms which remain after subtracting the leading order at R of the non-remainder terms in (\ref{smoothexpnearr}) are precisely
\begin{equation}\label{someterms}
-L\log\eta+^{R}Tr H^{Z}(\eta^{-2})-f_{\infty}.\end{equation}
We now plug these terms into (\ref{shorttimezeta}); we could switch (\ref{shorttimezeta}) to an integral in $\eta$, but it is easier to switch (\ref{someterms}) back to $(\epsilon,\tau)$, and it becomes
\[L\log(\sqrt\tau)+^{R}Tr H^{Z}(\tau)-f_{\infty}.\]
Plugging these into (\ref{shorttimezetatau}) and integrating from $\sqrt\tau=b$ to $\sqrt\tau=1/\epsilon$ gives a contribution of
\begin{equation}\label{thatcontrib}
\frac{2\epsilon^{2s}}{\Gamma(s)}\int_{b}^{1/\epsilon}
(L\log(\sqrt\tau)+^{R}Tr H^{Z}(\tau)-f_{\infty})\sqrt\tau^{2s-1}\ d\sqrt\tau.
\end{equation}

To get rid of the artificial upper bound of $1/\epsilon$, examine
\begin{equation}\label{thiscontrib}
\frac{2\epsilon^{2s}}{\Gamma(s)}\int_{1/\epsilon}^{\infty}
(L\log(\sqrt\tau)+^{R}Tr H^{Z}(\tau)-f_{\infty})\sqrt\tau^{2s-1}\ d\sqrt\tau.
\end{equation}
Because of the positive-order decay of (\ref{someterms}) at $\tau=\infty$, the integral in
(\ref{thiscontrib}) is well-defined in a neighborhood of $s=0$ and has
uniform-in-$s$ decay to a positive order in $\epsilon$. The same is true
of its $s$-derivative, precisely as in subsection \ref{leadingorderterm}. 
Therefore, (\ref{thiscontrib}) contributes
only terms of order $o(1)$ to the determinant, and we can replace
the analysis of (\ref{thatcontrib}) by analysis of
\begin{equation}\label{finalterm}
\frac{2\epsilon^{2s}}{\Gamma(s)}\int_{b}^{\infty}
(L\log(\sqrt\tau)+^{R}Tr H^{Z}(\tau)-f_{\infty})\sqrt\tau^{2s-1}\ d\sqrt\tau.
\end{equation}
We can directly evaluate most of the integrals in (\ref{finalterm}),
and we obtain a contribution of
\begin{equation}\label{finalcontrib}
\frac{-L\epsilon^{2s}b^{2s}}{2s\Gamma(s)}[2\log b-\frac{1}{s}]
+\frac{2\epsilon^{2s}}{\Gamma(s)}\int_{b}^{\infty}\ ^{R}Tr H^{Z}(\tau)
\sqrt\tau^{2s-1}\ d\sqrt\tau+\frac{f_{\infty}\epsilon^{2s}b^{2s}}{s\Gamma(s)}.
\end{equation}.

As for the remainder, since both (\ref{smoothexpnearr}) and the non-remainder terms in (\ref{smoothexpnearr}) are phg conormal on $Q_0$, so is $\tilde R$; moreover, it has positive leading order at F and leading order 0 at R. When we subtract the constant term at R, $\tilde R$ has positive-order decay at both L and R, hence at both $\eta=0$ and $\sqrt t=0$. As in the analysis of the remainder term in subsection \ref{theerrorterm}, $\tilde R$ only contributes a term of $o(1)$ to the zeta function and determinant.

Combining all the analysis in this subsection, we see that the
total contribution of the $\epsilon Z$ term to the zeta function, after subtracting the leading
orders at R, is
\[\sum_{k=0}^{n-1}\frac{2D_{k}\epsilon^{2s+k-n}b^{2s+k-n}}{(2s+k-n)\Gamma(s)}-
\frac{L\epsilon^{2s}b^{2s}\log\epsilon}{s\Gamma(s)}
-\frac{Ll_{log}\epsilon^{2s}b^{2s}}{s\Gamma(s)}\]\begin{equation}
\label{ezcontrib}
-\frac{L\epsilon^{2s}b^{2s}}{2s\Gamma(s)}[2\log b-\frac{1}{s}]
+\frac{f_{\infty}\epsilon^{2s}b^{2s}}{s\Gamma(s)}+
\frac{\epsilon^{2s}}{\Gamma(s)}\int_{0}^{\infty}\ ^{R}Tr H^{Z}(\tau)\tau^{s-1}\ 
d\tau,\end{equation}
modulo terms which contribute only $o(1)$ to the determinant expansion.

We now recognize the last term above as precisely 
$\epsilon^{2s}(^{R}\zeta_{Z}(s))$ - that is, $\epsilon^{2s}$ times the
renormalized zeta function on $Z$. To analyze this term, recall that the renormalized heat trace $^RTr H^Z(\tau)$ has polyhomogeneous expansions in $\tau$ at $\tau=0$ and $\tau=\infty$. From Proposition \ref{indexsetatzero}, the expansion at $\tau=0$ has index set in terms of $\sqrt\tau$ given by $-n+2\mathbb N_0$, similar to the usual heat kernel on a compact manifold; hence \[\int_{0}^{1}\ ^{R}Tr H^{Z}(\tau)\tau^{s-1}\ d\tau\] will have no pole at $s=0$. However, there is a term $-L\log\sqrt\tau$ in the expansion at $\tau=\infty$, and as a result the renormalized zeta function $^R\zeta_Z(s)$ may have a pole at $s=0$. In particular, we may write the Laurent series for $^R\zeta_Z(s)$ about $s=0$ as
\[^R\zeta_Z(s)=-\frac{L}{4}s^{-1}+^R\zeta_Z(0)+^R\zeta_Z'(0)\cdot s+\mathcal O(s^2).\]

\subsection{Proof of Theorem \ref{detapprox}}
We now combine 
(\ref{leadingproj}), 
(\ref{omeganoughtcontrib}), and (\ref{ezcontrib}), treating the $k=n$ case
of (\ref{leadingproj}),
(\ref{omeganoughtcontrib}), and (\ref{ezcontrib}) separately. We also use the fact, from \cite[Theorem 2.1]{ch2}, that $K=-L/2$, which yields major cancellations.
The zeta
function $\zeta_{\Omega_{\epsilon}}(s)$, modulo terms which contribute only
$o(1)$ to the determinant, is:

\begin{equation}\label{finalmess}\zeta_{\Omega_{0}}(s)+\frac{1}{\Gamma(s)}
\sum_{k=0}^{n-1}\frac{(-a_{k}+\tilde a_{k}+D_{k})
{b \epsilon}^{2s+k-n}}{s+(k-n)/2}
+\frac{\epsilon^{2s}b^{2s}}{s\Gamma(s)}(-a_{n}+\tilde a_{n}-Ll_{\log}
+f_{\infty})+
\epsilon^{2s}(^{R}\zeta_{Z}(s)).\end{equation}

This expression is a meromorphic function of $s$, uniformly in $\epsilon$. We may therefore
evaluate the $s$ coefficient of its Laurent series at $s=0$. When we do this,
we must take derivatives of the gamma function, so we 
write the Laurent series for $\frac{1}{\Gamma(s)}$ as $s+A_{2}s^{2}+
A_{3}s^{3}+\ldots$. We obtain
\[\zeta'_{\Omega_{\epsilon}}(0)=
\zeta'_{\Omega_{0}}(0)-\frac{L}{2}(\log\epsilon)^2+^{R}\zeta_{Z}(0)(2\log\epsilon)+^{R}\zeta_{Z}'(0)
+\sum_{k=0}^{n-1}\frac{(-a_{k}+\tilde a_{k}+D_{k})b^{k-n}\epsilon^{k-n}}
{(k-n)/2}\]
\[+(2\log b+2\log\epsilon +A_{2})(-a_{n}+\tilde a_{n}-Ll_{\log}
+f_{\infty})+o(1).\]
This equation appears quite complicated. However, the choice of $b$ was
arbitrary, so the coefficients of all powers of $\epsilon$ must be independent
of $b$. This allows successive simplifications.

First notice that the coefficient of $\epsilon^{k-n}$, for each $k$ between
0 and $n-1$, is a constant times $b^{k-n}$. This coefficient must be
independent of $b$, so the constant must be zero\footnote{
This cancellation
may seem rather fortunate, but it just reflects the artificial nature of the 
decomposition at $\sqrt\tau=b$.}. Therefore
\[\zeta'_{\Omega_{0}}(0)-\frac{L}{2}(\log\epsilon)^2+^{R}\zeta_{Z}(0)(2\log\epsilon)+^{R}\zeta_{Z}'(0)+
(2\log b+2\log\epsilon +A_{2})(-a_{n}+\tilde a_{n}-Ll_{\log}+f_{\infty}),\]
plus terms vanishing as $\epsilon$ goes to zero.
Finally, the constant term in this expansion is
\[\zeta'_{\Omega_{0}}(0)+^{R}\zeta_{Z}'(0)+
(2\log b+A_{2})(-a_{n}+\tilde a_{n}-Ll_{\log}+f_{\infty}),\]
and again it must be independent of $b$. Therefore
$-a_{n}+\tilde a_{n}-Ll_{\log}+f_{\infty}=0$.
We can now write
\[\zeta'_{\Omega_{\epsilon}}(0)=-\frac{L}{2}(\log\epsilon)^2+
2\log\epsilon(^{R}\zeta_{Z}(0))+\zeta'_{\Omega_{0}}(0)+^{R}\zeta_{Z}'(0)
+o(1).\]
Multiplying by $-1$, using the definition of $L$, and using the definition of the determinant of the
Laplacian completes the proof of Theorem \ref{detapprox}.

As a quick check on Theorem \ref{detapprox}, suppose that $Z=\mathbb R^n$; this is the trivial case, where $\Omega_0$ is flat in a neighborhood of P, with no conic singularity, and $\Omega_\epsilon=\Omega_0$ for every $\epsilon$. Then it is easy to compute the renormalized heat trace; it is the finite part as $\epsilon\rightarrow 0$ of
\[\int_{r\leq 1/\epsilon}H^{\mathbb R^n}(t,z,z)\ dz=\frac{Vol(r\leq 1/\epsilon)}{(4\pi t)^{n/2}}=\epsilon^{-n}(4\pi t)^{-n/2}Vol(r\leq 1).\]
The finite part is precisely zero, which means that the renormalized heat trace is identically zero, and hence the same is true for the renormalized zeta function and determinant. Note also that all heat invariants on $\mathbb R^n$ beyond the first are zero. Theorem \ref{detapprox} then reduces to $\zeta'_{\Omega_\epsilon}(0)=\zeta'_{\Omega_0}(0)+o(1)$, as expected (in fact, the error is zero in this case).


\section{The parametrix construction}

In this section, we construct the heat kernel for $\Omega_{\epsilon}$ 
by gluing together the heat kernels on $\epsilon Z$ and $\Omega_{0}$. 
We then use the construction to 
analyze the fine structure of the heat trace and prove
Theorem \ref{structure}.

\subsection{Outline of the parametrix construction}
In order to do a gluing construction, we need to introduce cutoff functions. First restrict attention to
$0<\epsilon<1/2$, and note that $\Omega_{\epsilon}$ is exactly conic
between $r=\frac{\epsilon}{2}$ and $r=2$. Recall from the introduction that $\chi_{1}$ is a smooth radial cutoff function on $\Omega_{\epsilon}$, equal to 1 on $r\leq 15/16$
and 0 on $r\geq 17/16$, and that $\chi_{2}=1-\chi_{1}$. 
Define two slightly larger radial cutoff functions; let $\tilde\chi_1$ be 1 on $r\leq 9/8$
and 0 on $r\geq 5/4$. Let
$\tilde\chi_{2}$ be 0 on $r\leq 3/4$ and $1$ on $r\geq 5/8$. These cutoffs are
illustrated, as functions of $r$, in Figure 6.1.

\begin{figure}\label{cutoffs}
\centering
\includegraphics[scale=0.70]{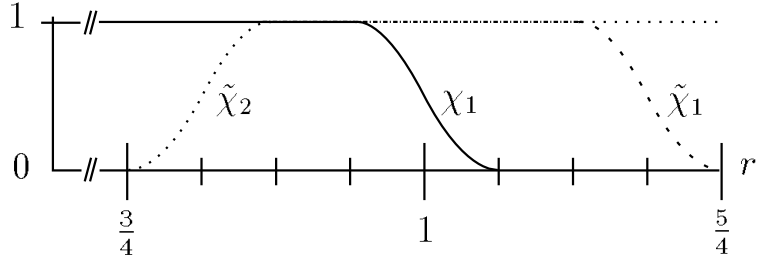}
\caption{Cutoffs as functions of $r$.}
\end{figure}
We can now define the initial parametrix:
\begin{definition} The \emph{parametrix} $G_\epsilon(t)$ is the $t$-dependent operator on $\Omega_\epsilon$ with Schwarz kernel
\begin{equation}\label{parametrix}
G_{\epsilon}(t,z,z')=\tilde\chi_{1}(z)H^{\epsilon Z}(t,z,z')\chi_{1}(z')
+\tilde\chi_{2}(z)H^{\Omega_{0}}(t,z,z')\chi_{2}(z').\end{equation}\end{definition}

Our construction is similar to 
the usual construction of the heat kernel, which is
explained, for example, in \cite{r}.
Denote by $*$ the operation of $t$-convolution: if two 
time-dependent operators $A$ and $B$ on a manifold $M$
have Schwarz kernels $A(t,z,z')$ and $B(t,z,z')$ respectively, and the integral
\begin{equation}\label{convsk}\int_{0}^{t}\int_{M}A(t-s,z,z'')B(s,z'',z')\ dz''\ ds.
\end{equation} converges for all $t$, then the $t$-convolution $A*B$ is well-defined with Schwarz kernel given by (\ref{convsk}).
It is easy to check that $t$-convolution is associative; denote by $A^{*k}$
the operator $A$, $t$-convolved with itself $k$ times. It is also easy to
check that $Tr (A*B)(t)=Tr (B*A)(t)$ whenever both $A*B$ and $B*A$ are trace class.

We first define the error $E_{\epsilon}(t,z,z')$ as
a kernel on $\Omega_{\epsilon}$:
\begin{equation}\label{error}
E_{\epsilon}(t,z,z')=(\partial_{t}+\Delta_{\Omega_{\epsilon},z})G_{\epsilon}(t,z,z').
\end{equation}
Then define
\begin{equation}\label{neumann}
K_{\epsilon}(t,z,z')=\sum_{k=1}^{\infty}(-1)^{k+1}E_{\epsilon}^{*k}(t,z,z').
\end{equation} 
We claim:
\begin{proposition}\label{hkform} For each $\epsilon$,
$H^{\Omega_{\epsilon}}(t,z,z')=G_{\epsilon}(t,z,z')
-(G_{\epsilon}*K_{\epsilon})(t,z,z')$.\end{proposition}

\begin{proof}
The proof of Proposition \ref{hkform} is relatively standard, and is based
on \cite{r}. We first notice that for any smooth
function $f(z')$ on $\Omega_{\epsilon}$,
\[\lim_{t\rightarrow 0}\int_{\Omega_\epsilon}G_{\epsilon}(t,z,z')f(z')\ dz'=f(z).\]
This is an immediate consequence of (\ref{parametrix}) and the definition
of the heat kernels on $\epsilon Z$ and $\Omega_{0}$. Further, we claim
that for each fixed $\epsilon$ and each $n$, there is a $C_{n,\epsilon}$ so that $|A_{\epsilon}(t,z,z')|<C_{n,\epsilon}t^{n}$ for
small $t$, where $A$ is any of $E$, $K$, or $G*K$; this is a consequence
of Lemmas \ref{errorlemma} and \ref{decay}.
Assuming these two claims, and that all convolutions are well-defined, 
we compute (suppressing $\epsilon$ in the notation):

\[(\partial_{t}+\Delta_{z})(G-G*K)(t,z,z')=E(t,z,z')
-(\partial_{t}+\Delta_{z})(G*K)(t,z,z')\]\[
=E(t,z,z')-\lim_{s\rightarrow t}\int_{\Omega_{\epsilon}}G(t-s,z,z'')K(s,z'',z')\
dz'-(E*K)(t,z,z')\]\begin{equation}\label{thisthingone}
=E(t,z,z')-K(t,z,z')-(E*K)(t,z,z')=0.\end{equation}

Moreover, for any smooth function $f(z')$ on $\Omega_{\epsilon}$,
\begin{equation}\label{thisthingtwo}
\lim_{t\rightarrow 0}\int_{\Omega_{\epsilon}}(G-G*K)(t,z,z')f(z')
=f(z)-\int_{\Omega_{\epsilon}}\lim_{t\rightarrow 0}(G*K)(t,z,z')f(z')\ dz'=f(z).\end{equation}
Uniqueness of the heat kernel completes the proof of Proposition \ref{hkform},
modulo the proofs of Lemmas \ref{errorlemma} and \ref{decay}.
\end{proof}

Proposition \ref{hkform} allows us to write
\[Tr H_{\epsilon}(t)=Tr G_{\epsilon}(t) -Tr (G_{\epsilon}*K_{\epsilon})(t)
=Tr G_{\epsilon}(t)-Tr (K_{\epsilon}*G_{\epsilon})(t).\]
In fact, we notice that $K_{\epsilon}=E_{\epsilon}-(K_{\epsilon}*E_{\epsilon})$,
and so \begin{equation}\label{threeterms}
Tr H_{\epsilon}(t)=Tr(G_{\epsilon}(t))-Tr((E_{\epsilon}*G_{\epsilon})(t))+
Tr((K_{\epsilon}*(E_{\epsilon}*G_{\epsilon}))(t)).\end{equation}
We will analyze each of the three terms and prove
that each is polyhomogeneous conormal on $Q_{0}$; (\ref{structure})
follows immediately. Throughout, we use the local coordinates on $Q_0$ from the previous sections; see Figure \ref{qnaught}.

\subsection{Heat kernels on $\epsilon Z$ and $\Omega_0$}
We first need to understand the asymptotic structure of the heat kernels of $\epsilon Z$ and $\Omega_0$. 
Since $\Omega_0$ is a manifold with a single conic singularity, its heat kernel has been extensively studied. As discussed in section 3, there are short-time heat trace asymptotics due to Cheeger and Bruning-Seeley \cite{ch2,brs}. However, for our parametrix construction, we need a complete description of the asymptotic structure of $H^{\Omega_0}(t,z,z')$ for short time; this description is provided by the work of Mooers \cite{mo}. Near the conic singularity, we use the polar coordinates $z=(r,y)$ and $z'=(r',y')$, so that the heat kernel is a function on
$\mathbb R_+^3(\sqrt t,r,r')\times N_y\times N_y'$. Let $M(t,r,r',y,y')$ be an iterated blow-up of this space: first blow up $\{r=r'=\sqrt t=0\}$, and then blow up the closure of the lift of $\{\sqrt t=0,r=r'>0\}$ (this process is illustrated in Figure \ref{hkconicfig}). Mooers then proved that:

\begin{figure}
\centering
\includegraphics[scale=0.65]{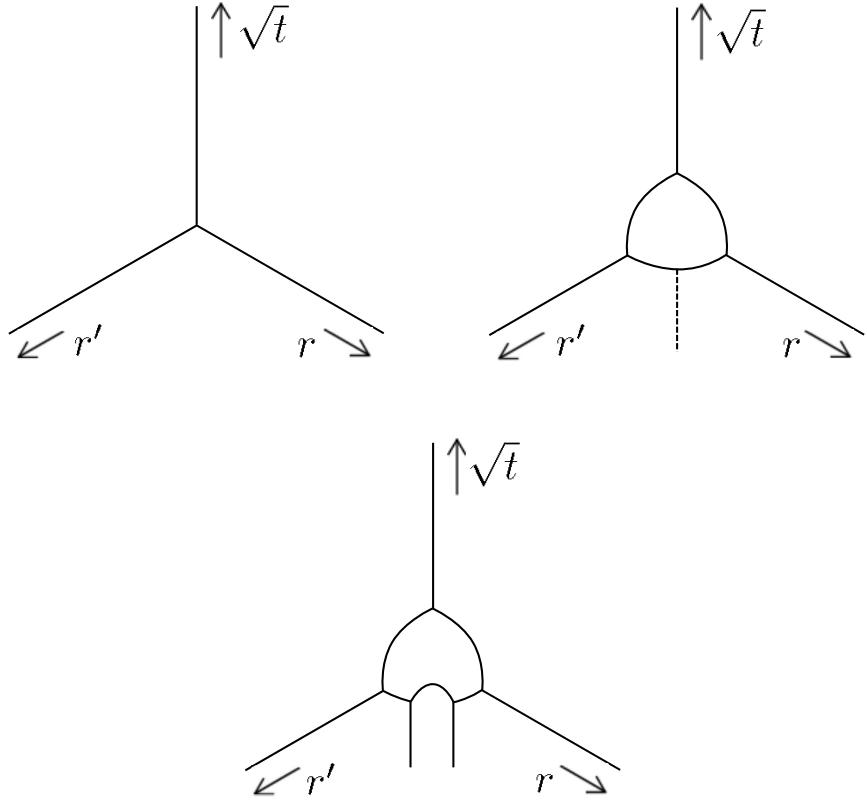}
\caption{Asymptotic structure of the heat kernel on a conic manifold.}\label{hkconicfig}
\end{figure}

\begin{proposition}\label{conehkstructure}
\cite{mo} If $\Omega_{0}$ is any manifold with a conic
singularity, and $(r,y)$ are polar coordinates on $\Omega_{0}$ near the
singularity, then $H^{\Omega_{0}}(t,r,r',y,y')$ is polyhomogeneous
conormal on $M(\sqrt t,r,r',y,y')$. The leading orders are $-n$ (in terms of $\sqrt t$) at the two blown-up faces, $\infty$ at the original $\sqrt t=0$ face, and 0 at the other two faces. \end{proposition}

As for the heat kernel on $\epsilon Z$, we need to understand $H^{\epsilon Z}(t,z,z')$ for bounded $t$. However, by the parabolic scaling of the heat kernel, we have
\[H^{\epsilon Z}(t,z,z')=\epsilon^{-n}H^Z(\frac{t}{\epsilon^2},\frac{z}{\epsilon},\frac{z'}{\epsilon}).\]
Therefore, to understand $H^{\epsilon Z}$ for bounded time, uniformly as $\epsilon$ goes to zero, we need to understand $H^Z$ for both short and long time. The short-time structure is relatively well-understood as a result of the work of Albin \cite{alb}, and in \cite{s1} we investigate the long-time structure by studying the low-energy resolvent on $Z$. 

We now summarize these results. On $Z$, let $x=r^{-1}$, so that $x$ is a 'boundary defining function' for infinity. Then the heat kernel is a function on $(0,\infty)_t\times\mathbb R_+^2(x,x')\times N_y\times N_{y'}$. Assume first that $t<T$ for some fixed positive $T$; we construct the appropriate space in a series of steps. We form the 'scattering double space' $Z^2_{sc}(x,y,x',y')$ by starting with $\mathbb R_+^2(x,x')\times N_y\times N_{y'}$ and performing two blow-ups. First blow up $\{x=x'=0\}$. Then let $D$ be the 'lifted diagonal' in this blown-up space given by the closure of the lift of the interior diagonal $\{x=x'>0,y=y'\}$, and blow up the intersection of $D$ with the boundary $x=x'=0$; the result is $Z^2_{sc}(x,y,x',y')$, and we call the face formed in this step sc. Finally, form $Z^2_{sc}(x,y,x',y')\times[0,T)_t$, and blow up the $t=0$ diagonal given by the closure of $\{t=0,x=x'>0,y=y'\}$; the face created by the final blow-up is called ff. We call the final space $Z^2_{sh}$ (for `short-time heat'); note that $Z^2_{sh}=[Z_{sc}^{2}(x,y,x',y')\times[0,1]_{\sqrt t}; \{\sqrt t=0,x=x',y=y'\}$. This space was first identified by Albin in \cite{alb} as the likely candidate for the polyhomogeneous structure of the short-time heat kernel on $Z$. In \cite{s1}, we prove as an immediate consequence of \cite{alb} that:
\begin{theorem}\label{zshorttime} For $t<1$, the heat kernel
$H^{Z}(t,x,y,x',y')$ is polyhomogeneous conormal on $Z^2_{sh}$. Moreover, there is infinite-order
decay at all faces except sc, where the leading order is 0, and ff, where the leading order is $\sqrt t^{-n}$.
\end{theorem}

On the other hand, assume that $t>T$ for some fixed positive $T$, and let $w=t^{-1/2}$, so that the heat kernel is a function of $(w,x,x',y,y')$. We define a blown-up space $Z^2_{w,sc}$ as follows. First, begin with the space $X_b^3(w,x,x')\times N_y'\times N_{y'}$; $X_b^3$ is a 'b-stretched product' in the terminology of Melrose and Singer \cite{ms}. More details may be found in \cite{gh1} or \cite{s1}. As before, let $D$ be the closure of the lift of $\{x=x'>0,y=y'\}$, then blow up the intersection of $D$ with the face $\{w>0,x=x'=0\}$ to obtain $Z^2_{w,sc}$. $Z^2_{w,sc}$ has eight boundary hypersurfaces and is illustrated in Figure \ref{ghspace}; the faces are also labeled as indicated in the figure. Let $F(w,x,y,x',y')=H^Z(w^{-2},x,y,x',y')$. In \cite{s1}, we prove the following.

\begin{figure}
\centering
\includegraphics[scale=0.45]{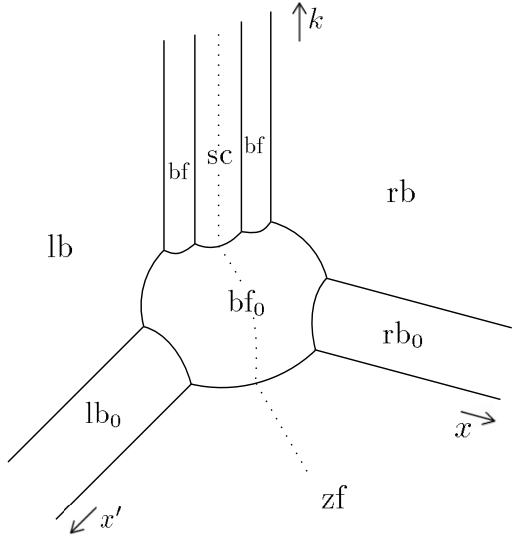}
\caption{The space $Z^{2}_{k,sc}$.}\label{ghspace}
\end{figure}

\begin{theorem}\label{maincor} For any $n\geq 2$, and for any fixed time $T>0$, $F(w,x,y,x',y')$ is polyhomogeneous conormal on $Z^2_{w,sc}$ for $t>T$, where $w=t^{-1/2}$. The leading orders at the boundary hypersurfaces are at least 0 at sc and $n$ at each of bf$_0$, rb$_0$, lb$_0$, and zf, with infinite-order decay at lb, rb, and bf.
\end{theorem}

Combining the short- and long-time structure theorems, we obtain a complete description of the asymptotic structure of the heat kernel on $Z$, illustrated in Figure \ref{hkac}; we have indicated the leading order at each face (in terms of $t^{1/2}$ or $w=t^{-1/2}$). We call this space $Z^2_{heat}$. In fact, as explained in \cite{s1}, both of these theorems hold even if $Z$ is only asymptotically conic rather than exactly conic near infinity.

Finally, we would like to be able to generalize the construction in this section to the case of the heat kernel for the Laplacian acting on sections of certain vector bundles over $Z$. In particular, we are interested in the twisted $k$-form bundles needed to define the analytic torsion. Based on work of Guillarmou and Hassell in \cite{gh2}, together with as-yet-unpublished work of Guillarmou \cite{g3}, we expect that in many cases, the heat kernel for forms will have the same polyhomogeneity properties as above, but with different leading orders. In particular, the leading orders of the heat kernel for long time may be worse at rb$_0$, lb$_0$, and zf. Therefore, we do not use the order $n$ decay at those faces that we see in the case of the scalar Laplacian; instead, we use only the following strictly weaker assumption, which reflects the expected features of the problem for the twisted form bundles:
\begin{althyp} The heat kernel on $Z$ has leading order at least 0 at zf and at least $n-a$ at rb$_0$ and lb$_0$, where $a\in[0,n/2)$. However, the coefficient of the term of order 0 at zf decays to order strictly greater than $n$ at bf$_0$.
\end{althyp}

\begin{figure}
\centering
\includegraphics[scale=0.80]{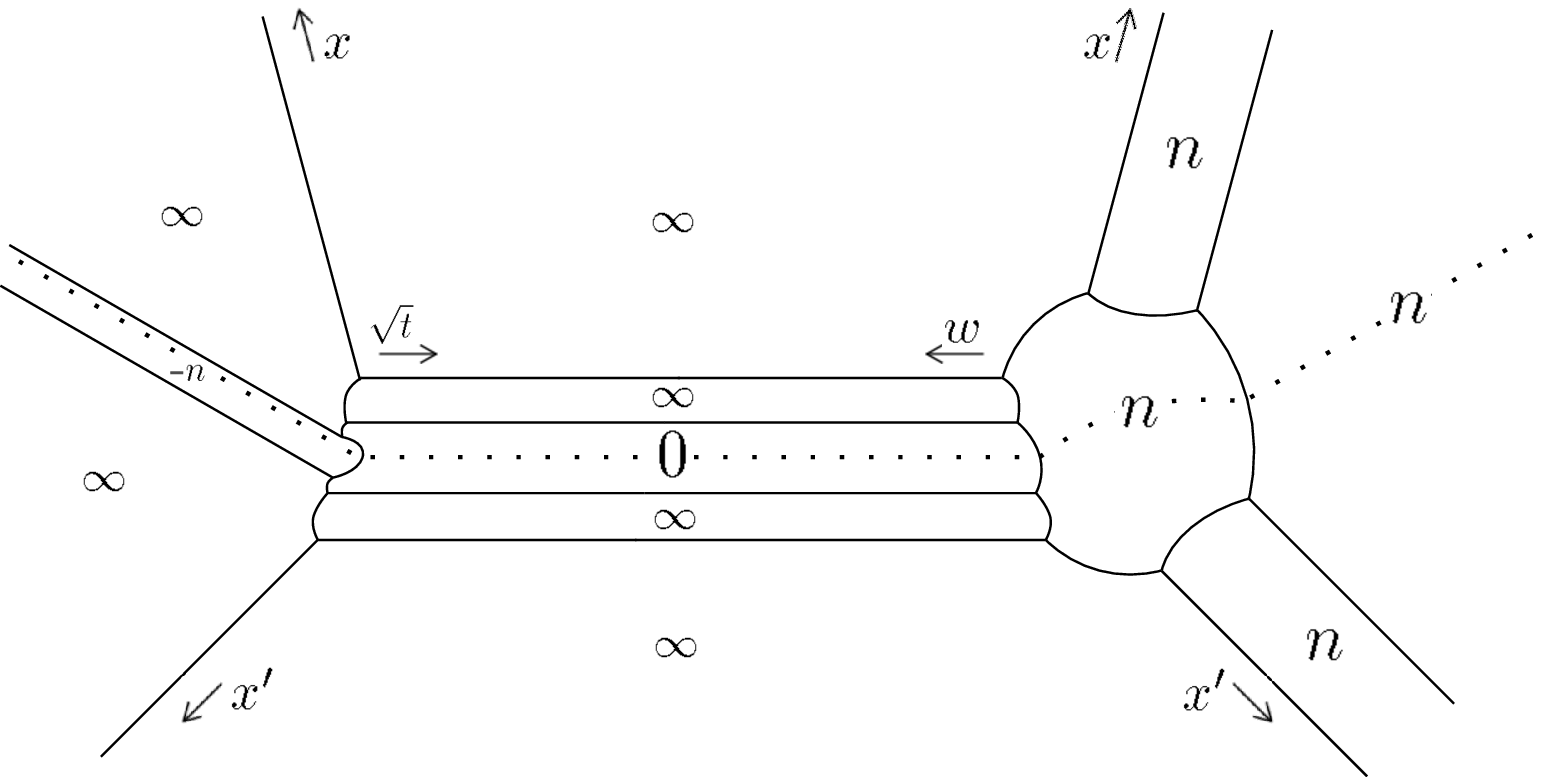}
\caption{The space $Z^2_{heat}$: asymptotic structure of the heat kernel on $Z$.}\label{hkac}
\end{figure}

\subsection{The error terms $E_{\epsilon}$ and $K_{\epsilon}$}

To begin proving Proposition \ref{hkform} (using only the alternate hypothesis), we first compute $E_{\epsilon}$. 
Since $H^{\epsilon Z}$ and $H^{\Omega_{0}}$ solve the heat equation
on the regions $r\leq 3/2$ and $r\geq 1/2$ respectively, we obtain:
\[E_{\epsilon}(t,z,z')=[\Delta^{\Omega_{\epsilon}}_{z},\tilde\chi_{1}(z)]
H^{\epsilon Z}(t,z,z')\chi_{1}(z')
+[\Delta_{\Omega_{\epsilon},z},\tilde\chi_{2}(z)]H^{\Omega_{0}}(t,z,z')\chi_{2}(z').\]
The commutators are first-order differential operators, with supports
contained in the $r$-intervals $[9/8,5/4]$ and $[3/4,7/8]$ respectively. 
In these regions, $\Omega_{\epsilon}$ is exactly conic for all $\epsilon<1/2$, 
so we may call the commutators $V_{1}(z)$ and $V_{2}(z)$. Then
\begin{equation}\label{writtenerror}
E_{\epsilon}(t,z,z')=V_{1}(z)H^{\epsilon Z}(t,z,z')\chi_{1}(z')+
V_{2}(z)H^{\Omega_{0}}(t,z,z')\chi_{2}(z').\end{equation}
Note that $V_{1}(z)$ is supported in $(9/8,5/4)$, while
$\chi_{1}(z')$ is supported in $r\leq 17/16$; similarly, $V_{2}(z)$ is
supported in $(3/4,7/8)$, and $\chi_{2}(z')$ is supported in $r\geq 15/16$.
Therefore, the heat kernels in the definition of $E_{\epsilon}$
are each supported away from the diagonal. Off-diagonal
heat kernels usually decay to infinite order at $t=0$, which motivates
the following lemma. Recall from section 2 that $R_{\epsilon}(z)$ is the weight function on $\Omega_\epsilon$ which is $\epsilon$ in $\{r\leq\epsilon\}$, 1 outside $\{r\leq1\}$, and $r$ in between; define a modification $W_\epsilon(z)$ to be $r$ between $\{r=\epsilon\}$ and $\{r=3/2\}$ and $3/2$ outside $\{r\leq 3/2\}$. Then:

\begin{lemma}\label{errorlemma} For each $N\geq 0$ and $t$ less than any fixed $T$, there is a constant $C_N$ such that for all $\epsilon<1/2$, \[\sup_{z,z'\in\Omega_{\epsilon}}
|(W_{\epsilon}(z'))^{a}E_{\epsilon}(t,z,z')|\leq C_Nt^{N}.\]
Here the constant $a$ is as in the Alternative Hypothesis. Moreover, for each fixed $z$ and $z'$ both outside
$r=1/2$, $E_{\epsilon}(t,z,z')$ is
phg conormal on $Q_{0}$, and the dependence on $z$ and $z'$ is
smooth. \end{lemma}

We require $r\geq 1/2$ in the second part of the statement in order to
regard $z$ and $z'$ as coordinates on the fixed space $\Omega_{0}$.
We now claim:

\begin{corollary} For each fixed $z$ and $z'$ both outside $r=1/2$,
$E_{\epsilon}(t,z,z')$ is phg conormal on $Q$ with infinite-order
decay at $\sqrt t=0$.
\end{corollary}

The corollary follows from Lemma \ref{errorlemma} and the general blow-up fact, from \cite{me2}, that any function phg conormal on $[M;P]$ which vanishes to infinite order at the blown-up face is in fact phg conormal on $M$ without the blow-up (see Proposition A.3 of \cite{s}). We now prove Lemma \ref{errorlemma}.

\begin{proof} 

First consider the second term in $E_{\epsilon}$,
$V_{2}(z)H^{\Omega_{0}}(t,z,z')\chi_{2}(z')$. It 
is independent of $\epsilon$, and its support in $(z,z')$
is away from the conic tip in both variables and also away from the
spatial diagonal. By Proposition \ref{conehkstructure}, for such $(z,z')$,
$H^{\Omega_{0}}(t,z,z')$ is phg conormal with infinite-order decay
at $\sqrt t=0$, uniformly in the spatial variables; so $V_{2}(z)H^{\Omega_{0}}$
has the same properties. This proves the phg conormality part of the lemma,
and the sup-norm statement is automatic since $\Omega_{0}$ is compact.

It remins to consider $V_{1}(z)H^{\epsilon Z}(t,z,z')\chi_{1}(z')$. By
the scaling property of the heat kernel,
\[V_{1}(z)H^{\epsilon Z}(t,z,z')\chi_{1}(z')=\epsilon^{-n}V_{1}(z)H^{Z}(
t/\epsilon^{2},z/\epsilon,
z'/\epsilon)\chi_{1}(z').\]
We now use the coordinates $(x,y,x',y')$. On
$\epsilon Z$,
$V_{1}(z)$ is a bounded linear combination 
of derivatives $\partial_{r}$ and $\partial_{y}$,
which are equivalent to $r\partial_{r}$ and $\partial_{y}$
since $r\in[9/8,5/4]$. These lift to derivatives of the form $x\partial_{x}$ and
$\partial_{y}$ on $Z$, so $V_{1}$ lifts to a $b$-differential operator on $Z$, which
we call $V_{1}$ as well.

\medskip

First we prove the phg conormality statement in the lemma. Fix
an $x$ and $x'$ with $x\in (4/5,8/9)$ and $x'\in (16/17,2)$ and also fix $y$ and
$y'$. Of course, $\epsilon^{-n}$ is phg conormal on $Q_{0}$, so
we only need to analyze
\[V_{1}H^{Z}(t/\epsilon^{2},\epsilon x, y, \epsilon x', y').\] 

Let $S\subset Z^2_{heat}$ be the closure in $Z^2_{heat}$ of the set of points of the form $(s,a, y,a(x'/x), y')$ for some $s,a>0$. Since $x$ and $x'$ are fixed and $x'/x$ is bounded away from 1, $S$ is a p-submanifold\footnote{See \cite[Ch. 2A]{ma} for a definition of p-submanifolds of manifolds with corners and discussion of their properties. Essentially, they are submanifolds which may be written as coordinate submanifolds near the intersection with any boundary or corner.} of dimension 2 in $Z^2_{heat}$. Hence the heat kernel $H^Z$ is also polyhomogeneous conormal once we restrict to $S$; since $V_1$ is a b-differential operator, all the same analysis holds for $V_1H^Z$. 

Now let $\phi:Q_0\cup\{t\leq T\}\rightarrow Z^2_{heat}$ be given by the map $(t,\epsilon)\rightarrow (t/\epsilon^2,\epsilon x,y,\epsilon x',y')$. The image of $\phi$ is contained in $S$; since $t\leq T$, we have $\epsilon/(t/\epsilon^2)^{-1/2}\leq T$, so the image of $\phi$ is also bounded away from zf (see Figures \ref{ghspace} and \ref{hkac}). We would like to say that the pullback of $H^Z$ by the map $\phi$ is phg conormal. To do this, we need to check that $\phi$ is a \emph{b-map}.\footnote{See \cite{ma} or \cite{gri} for a definition and discussion of b-maps.} Once we have this, then the pullback theorem of Melrose (\cite{me3}; see \cite[Theorem 3.12]{gri}) will complete the proof.

In order to check that $\phi$ is a b-map, we need to compute the pullbacks of boundary defining functions for each boundary hypersurface of $S$ and check that they are products of boundary defining functions for the boundary hypersurfaces of $Q_0$. There are three of these hypersurfaces. For small $s$, the boundary defining functions are $s$ and $a$. For large $s$, the boundary defining functions are $s^{-1/2}$ and $a/s^{-1/2}=a\sqrt s$. Now the pullback of $s$ is $\tau$ and the pullback of $a$ is $\epsilon x$. For small $s$, the pullbacks are thus $\tau$ and $\epsilon x$, which are boundary defining functions for the faces L and F of $Q_0$ respectively. For large $s$, the pullbacks are $\tau^{-1/2}=\eta$ and $\epsilon\sqrt{\tau}=\sqrt t$, which are boundary defining functions for R and F respectively. Thus $\phi$ is a b-map and the proof is complete.

We must also check that the dependence on $z$ and $z'$ is smooth. However,
this follows immediately from the fact that derivatives $\partial_{r}$
and $\partial_{y}$ lift to b-derivatives on $Z$, as we have already
computed.

\medskip

Now we prove the supremum statement. Fix any $N\in\mathbb N$ and consider the function
\begin{equation}t^{-N}\epsilon^{-n}(W_\epsilon(z'))^aV_1 H^Z(\tau,\frac{z}{\epsilon},\frac{z'}{\epsilon})=t^{-N}\epsilon^{-n}(W_\epsilon(z'))^aV_1 H^Z(\tau,\epsilon x,y,\frac{z'}{\epsilon}).
\end{equation}We will show that this function is globally bounded for all $t<T$, all sufficiently small $\epsilon$, and all $(x,y,z')$ with $x\in(4/5,8/9)$ and $|z'|\leq 17/16$. Rewriting in terms of the arguments of $V_1 H^Z$, we have
\begin{equation}\label{lookatme}\tau^{-N}(\epsilon x)^{-2N-n}x^{-2N+n}(W_\epsilon(z'))^a
V_1 H^Z(\tau,\epsilon x,y,\frac{z'}{\epsilon}).\end{equation}
Note that, with our support condition, $W_{\epsilon}(z')$ is equal to $|z'|=x^{-1}(\epsilon x)|z'|/\epsilon$ when $|z'|/\epsilon\geq 1$ and $\epsilon=(\epsilon x)x^{-1}$ otherwise - i.e. $W_{\epsilon}(z')=x^{-1}(\epsilon x)\max\{1,|z'|/\epsilon\}$.

To show the global boundedness, we examine the behavior of (\ref{lookatme}) at each boundary hypersurface of $Z^2_{heat}(\tau,z/\epsilon,z'/\epsilon)$, showing that it is bounded at all such hypersurfaces. Global boundedness follows immediately by compactness. Note that by our support conditions, the arguments of $V_1 H^Z$ are bounded away from the faces ff, sc, rb, rb$_0$, and zf (the bound away from zf is because $t$ is bounded; note that $\tau^{-1/2}\geq\epsilon T^{-1/2}>\epsilon x T^{-1/2}$), so (\ref{lookatme}) is zero in a neighborhood of those faces. Moreover, $H^Z$ has infinite-order decay at all other faces for small $\tau$, as well as lb, bf, and rb. It remains only to check boundedness at bf$_0$ and lb$_0$.

At bf$_0$, the heat kernel $H^Z(\tau,\epsilon x,y,\frac{z'}{\epsilon})$ has leading order $n$ (and hence so does $V_1H^Z$, since $V_1$ is a b-differential operator). The factor of $\tau^{-N}$ is equal to $\eta^{2N}$ and hence it has order $2N$. The factor of $(\epsilon x)^{-2N-n}$ has order $-2N-n$, which precisely cancels the previous orders. And $W_{\epsilon}(z')$ is globally bounded by $3/2$ (in fact, its order is precisely 0 at bf$_0$ - $\epsilon x$ has order 1 at bf$_0$, but $|z'|/\epsilon$ has order $-1$). We conclude that (\ref{lookatme}) is bounded near bf$_0$. 

As for lb$_0$, the Alternative Hypothesis implies that $V_1H^Z$ has order $n-a$. The factor of $\tau^{-N}$ has order $2N$, and the factor of $(\epsilon x)^{-2N-n}$ has order $-2N-n$. Finally, at lb$_0$, note that both $1$ and $|z'|/\epsilon$ have order zero ($|z'|/\epsilon$ blows up at bf$_0$, but not at lb$_0$). Since $(\epsilon x)$ has order 1 at lb$_0$, $W_{\epsilon}(z')^a$ has order $a$ at lb$_0$. Therefore the total order of (\ref{lookatme}) at lb$_0$ is $(n-a)+2N+(-2N-n)+a=0$, which shows the boundedness at lb$_0$. This completes the proof. \end{proof}

Now that we have analyzed $E_{\epsilon}$, we form the Neumann series for
$K_{\epsilon}$ and prove analogous short-time bounds for $K_{\epsilon}$.
Fix $T>0$.

\begin{lemma}\label{kerrorlemma}
a) The expression $E_{\epsilon}^{*k}(t,z,z')$ 
is defined for each $k$, and 
the sum $K_{\epsilon}(t,z,z')=\sum_{k=1}^{\infty}(-1)^{k+1}E_{\epsilon}^{*k}(t,z,z')$ 
converges.

b) For $t<T$ and any $M\in\mathbb N$, there is a $C_{M}$ independent
of $\epsilon$ such that 
\[\sup_{z,z'\in\Omega_{\epsilon}}|(W_\epsilon(z'))^aK_{\epsilon}(t,z,z')|\leq Ct^{M}.\]

c) For each
$z$, $z'$ with $r\geq 1/2$, $K_{\epsilon}(t,z,z')$ is phg conormal
on $Q$, with smooth dependence on $z$ and $z'$
in this region. Moreover, let the index set for $E_{\epsilon}(z,z',t)$ at $\epsilon=0$ be
$\mathcal F=\{(s_{j},p_{j})\}$; then the index set for $K_{\epsilon}(z,z',t)$
at $\epsilon=0$ is $\mathcal F_{+}$, 
the set of all finite sums of elements of $\mathcal F$.

\end{lemma}

Note that $\mathcal F$ is bounded below by zero by Lemma \ref{errorlemma}, 
so $\mathcal F_{+}$ is an index set and is also bounded below by zero.

\begin{proof} $|E_{\epsilon}(z,z',t)|\leq C_{0}$ uniformly for $t$ in any bounded
interval and $z$, $z'\in\Omega_{\epsilon}$; therefore, all the integrals in
the convolutions converge, and $E_{\epsilon}^{*k}(t,z,z')$ is defined. To see that the sum over $k$ converges is slightly more involved. Since $\epsilon$ is
fixed for the moment, we suppress it in the notation.
All of the spatial integrands in the
convolutions are zero when the variable of integration is not in 
the support of $V_{1}$ or $V_{2}$ (we call this the 'cutoff region'), so we can
regard the integrals as being over a subset of $\Omega_{0}$. Let $V$ be the
volume of $\Omega_{0}$. Moreover, the weight function $W_\epsilon(z')$ is between $1/2$ and $1$ in the cutoff region.

Now for any $M$ and $t$ in any bounded interval $[0,T]$, Lemma \ref{errorlemma} implies that there is
a $C_{M}$ so that $|(W_\epsilon(z'))^aE(t,z,z')|<C_{M}t^{M}$.
By integrating and using the weight function bound on the cutoff region, we get that 
\[|E^{*2}(t,z,z')|=\left |\int_{0}^{t}\int_{\Omega_{0}}
E(t-s,z,z'')E(s,z'',z')\ dz''\ ds \right |\]
\[\leq\int_{0}^{t}2^aVC_{M}(t-s)^{M}C_{0}(W_\epsilon(z'))^{-a}\ ds\leq C_{M}t^{M}\frac{2^aVC_{0}t}{M+1}(W_\epsilon(z'))^{-a}.\]

We then use induction; we claim that $|E^{*k}(t,z,z')|
\leq C_{M}(M!)t^{M}\frac{(2^aVC_{0}t)^{k-1}}{(M+k-1)!}(W_\epsilon(z'))^{-a}$
for all $k$. Indeed, assuming the inductive hypothesis, we have:
\[|E^{*(k+1)}(t,z,z')|=\left |\int_{0}^{t}\int_{\Omega_{0}}E(t-s,z,z'')
E^{*k}(s,z'',z)\ dz''\ ds\right |\]
\[\leq \int_{0}^{t}(2^aVC_0)C_{M}(M!)s^{M}\frac{(2^aVC_{0}s)^{k-1}}{(M+k-1)!}(W_\epsilon(z'))^{-a}\ ds
=C_{M}(M!)t^{M}\frac{(2^aVC_{0}t)^{k}}{(M+k)!}(W_\epsilon(z'))^{-a}.\]

So: \[|K(t,z,z')|\leq\sum_{k=1}^{\infty}C_{M}(M!)t^{M}\frac{(2^aVC_{0}t)^{k-1}}{(M+k-1)!}(W_\epsilon(z'))^{-a}
\leq\sum_{k=0}^{\infty}C_{M}t^{M}\frac{(2^aVC_{0}t)^{k}}{k!}(W_\epsilon(z'))^{-a}\]\[=C_{M}t^{M}e^{2^aVC_{0}t}(W_\epsilon(z'))^{-a}.\]
This shows not only that the series converges, but also that
for $t<T$, $|K(t,z,z')|\leq\hat C_{M}t^{M}(W_\epsilon(z'))^{-a}$
for all $M$, where $\hat C_{M}=C_{M}e^{2^aVC_{0}T}$;
since $\hat C_{M}$ is independent of $\epsilon$, 
this completes the proof of parts a) and b).

Now assume that $z$ and $z'$ both have $r\geq 1/2$. 
From Lemma \ref{errorlemma} and the corollary following it, $E_{\epsilon}(t,z,z')$ is phg conormal on $Q$;
we want to show that $K_{\epsilon}(t,z,z')$ is also phg conormal on $Q$. Fix positive integers
$N$ and $b$. Since $E_{\epsilon}(t,z,z')$ is phg conormal, there
exist $l\geq 0$, $(s_{i},p_{i})\in E$ for each $i$ between 1 and $l$,
and coefficients $a_{i}(t,z,z')$ for each $i$ between 1 and $l$ 
\begin{equation}\label{eexp}
E_{\epsilon}(t,z,z')\sim\sum_{i=1}^{l}a_{i}(t,z,z')\epsilon^{s_{i}}
(\log\epsilon)^{p_{i}}+R_{N}(t,z,z',\epsilon).\end{equation}
Moreover, for each $M\in\mathbb N_{0}$, 
there exists $C_{b,M}$ such that $|a_{i}(t,z,z')|\leq C_{M}t^{M}$
and $R_{N}\leq C_{b,M}t^{M}\epsilon^{N}$, each together with up to $b$
derivatives of the form $(\epsilon\partial_{\epsilon})$ or $(t\partial_{t})$.

When we take convolutions
of $E_{\epsilon}$ with itself, we plug in (\ref{eexp}) and work term-by-term.
For each $(r,q)\in F_{+}$ with $\Re r\leq N$, and each
$k$, $E_{\epsilon}^{*k}(t,z,z')$ has a term of the form 
$b_{k,r,q}(t,z,z')\epsilon^{r}(\log\epsilon)^{q}$; $b_{k,r,q}$ may be zero.
We want to bound $|b_{k,r,q}(t,z,z')|$. There are at most $l^{k}$ combinations
of $(s_{i},p_{i})$ with $1\leq i\leq l$ which add up to $(r,q)$, and these
are the only combinations which contribute to $b_{k,r,q}$. For each
combination, we can perform the same convolution analysis as before to bound
the coefficient; we obtain that
\[|b_{k,r,q}(t,z,z')|\leq C_{b,M}t^{M}l^{k}\frac{(VC_{b,0}t)^{k-1}}{(k-1)!}.\]

The rest of the combinations yield terms of order $\epsilon^{N}$, as they
either involve a remainder or their $\epsilon$ exponents add up to more than $N$.
We combine all of these terms into a new remainder term $R_{N,k}(t,z,z',\epsilon)$.
There are at most $(l+1)^{k}$ of these combinations. We again perform the
convolution analysis to obtain that
\[|R_{N,k}(t,z,z',\epsilon)|\leq C_{b,M}t^{M}(l+1)^{k}\frac{(VC_{b,0}t)^{k-1}}
{(k-1)!}.\]
Moreover, let $\mathcal V$ be any differential operator 
of order $\leq b$ that is a combination
of $(\epsilon\partial_{\epsilon})$ and $(t\partial t)$ derivatives. We claim
that
\begin{equation}\label{theclaim}
|\mathcal V R_{N,k}(t,z,z',\epsilon)|\leq C_{b,M}t^{M}(l+1)^{k}k^{b}\frac{(VC_{b,0}t)^{k-1}}
{(k-1)!}.\end{equation}
Indeed, each of the $b$ derivatives in $\mathcal V$ may act on each of $k$
different terms of the form $a_{i}(t,z,z')\epsilon^{s_{i}}(\log\epsilon)^{p_{i}}$
or $R_{N}$. If they act on $R_{N}$, we still have the $C_{b,M}t^{M}\epsilon^{N}$
bound. On the other hand, if they act on $a_{i}(t,z,z')\epsilon_{s_{i}}
(\log\epsilon)^{p_{i}}$, the $\epsilon\partial_{\epsilon}$ derivatives
only improve the expansion, while the $t\partial_{t}$ derivatives hit $a_{i}(t,z,z')$,
and we can still apply the $C_{M}t^{M}$ bound. This proves (\ref{theclaim}).

Finally, we sum from $k=1$ to $\infty$. By the Ratio Test,
\[\tilde b_{r,q}=\sum_{k=1}^{\infty}|b_{k,r,q}(t,z,z')|\]
converges and is $\mathcal O(t^{M})$ for any $M$.
And for any differential operator $\mathcal V$ of order at most $b$ as above,
\[\tilde R_{N}=\sum_{k=1}^{\infty}|R_{N,k}|\] converges by the Ratio Test
and is also $\mathcal O(t^{M}\epsilon^{N})$. Therefore, we may write
\[K_{\epsilon}(t,z,z')=
\sum_{\{(r,q)\in\mathcal F_{+}, \Re r\leq N\}}\tilde b_{r,q}
\epsilon^{r}(\log\epsilon)^{q}+\tilde R_{N}.\]
Since $b$ and $N$ were arbitrary, 
$K_{\epsilon}(t,z,z')$ is therefore polyhomogeneous
conormal on $Q$, with index set $\mathcal F_{+}$ at $\epsilon=0$ and infinite-order
decay at $\sqrt t=0$. This completes the proof of Lemma \ref{kerrorlemma}.
\end{proof}

\subsection{Construction of the heat kernel}
We now analyze $G_{\epsilon}(t) * K_{\epsilon}(t)$. If we can show that
it is well-defined for each $\epsilon$, we will have shown that 
$G_{\epsilon}(t) - (G_{\epsilon}(t) * K_{\epsilon}(t))$ is in fact
the heat kernel $H^{\Omega_{\epsilon}}(t,z,z')$. The next lemma proves
what we need, and also gives a useful bound:

\begin{lemma}\label{decay}
Fix $T>0$. Let $A_{\epsilon}(t)$ be any family of operators on $\Omega_{\epsilon}$ 
whose Schwarz kernel $A_{\epsilon}(t,z,z')$ has $z$-support between $r=3/4$ and $r=5/4$, and suppose also that for each $M$ there is a $C_M$ such that for all $\epsilon<1/2$, all $z$, and all $t<T$,
\begin{equation}\label{boundform}|(W_\epsilon(z'))^aA_{\epsilon}(t,z,z')|\leq C_{M}t^{M}.\end{equation}
Then $(G_{\epsilon}*A_{\epsilon})(t,z,z')$ and $(A_{\epsilon}*G_{\epsilon})(t,z,z')$ are both
well-defined for each $\epsilon$. Moreover, for each $M$, there is a $\tilde C_M$ such that for all $\epsilon<1/2$, all $z$ and $z'$ in $\Omega_\epsilon$ with $z'$ between $r=3/4$ and $r=5/4$, and all $t<T$,
\begin{equation}\label{boundforma}|(A_{\epsilon}*G_{\epsilon})(t,z,z')|\leq \tilde C_{M}t^{M}.\end{equation}

\end{lemma}

In particular, this lemma may be applied with $A_\epsilon=E_\epsilon$ or $A_\epsilon=K_\epsilon$.

\begin{proof}

First we consider $A_{\epsilon}*G_{\epsilon}$;
$G_{\epsilon}(t)$ has two terms, which we consider separately.
The term coming from $\Omega_0$ gives
\begin{equation}\label{expressiononesix}
\int_{0}^{t}\int_{\Omega_0}A(s,z,z'')\tilde\chi_{2}(z'')H^{\Omega_0}(t-s,z'',z')
\chi_{2}(z')\ dz''\ ds.\end{equation}
Consider \begin{equation} h_2(s,z')=\int_{\Omega_0}\tilde\chi_{2}(z'')H^{\Omega_0}(s,z'',z')\chi_2(z')\ dz''.\end{equation}
By the definition of the heat equation, $h_2(s,z)$  is precisely $\tilde\chi_2(z')$ times the solution of the heat equation on $\Omega_0$ with initial data $\chi_2(z'')$. That solution approaches $\chi_2(z')$ as $s$ goes to zero and hence is uniformly bounded, for $s<T$, by a constant. So $h_2(s,z')$ is uniformly bounded for $s<T$. Combining this with (\ref{boundform}) yields both the well-definedness and the bound for this piece (since $z''$ is outside $r=1/2$, $W_\epsilon(z'')$ is between $1/2$ and $3/2$).

On the other hand, the term coming from $\epsilon Z$ is, substituting $\tilde s=t-s$,
\begin{equation}\label{expressiononefive}
\int_{0}^{t}\int_{\Omega_\epsilon}A(t-\tilde s,z,z'')\tilde\chi_{1}(z'')H^{\epsilon Z}(\tilde s,z'',z')
\chi_{1}(z')\ dz''\ d\tilde s.\end{equation}
We claim that
\begin{equation}\label{expressiontwo}
\int_{\Omega_\epsilon}\big |(W_\epsilon(z''))^{-a}\tilde\chi_{1}(z'')H^{\epsilon Z}(s,z'',z')\chi_{1}(z')\ dz''\big |\end{equation}
is bounded uniformly by some constant $C$ for all $z$, all $\epsilon<1/2$, and all $s<T$. Assuming the claim for the moment, the absolute value of (\ref{expressiononefive}) is bounded, for any $M$, by
\[C_Mt^M\int_0^t C\ ds=CC_Mt,\] which is more than enough.

As for $G_\epsilon*A_\epsilon$, we only need well-definedness and not a uniform-in-$\epsilon$ bound. So fix $\epsilon$ and break $G_\epsilon$ into two pieces. The $\Omega_0$ piece may be analyzed as above, switching $s$ for $t-s$ and $\tilde\chi_2$ for $\chi_2$. In the $\epsilon Z$ piece, we get precisely
\begin{equation}\int_{0}^{t}\int_{\epsilon Z}\chi_{1}(z)H^{\epsilon Z}(s,z,z'')
\tilde\chi_{1}(z'')A(t-s,z'',z')\ dz''\ ds.\end{equation}
Notice that
\begin{equation}h_1(s,\epsilon)=\int_{\epsilon Z}\chi_{1}(z)H^{\epsilon Z}(s,z,z'')
\tilde\chi_{1}(z'')\end{equation}
is $\chi_1(z)$ times the solution of the heat equation on the fixed manifold $\epsilon Z$ with initial data $\chi_1(z'')$. As before, the limit as $s$ goes to zero of this solution is precisely $\tilde\chi_1(z)$, and hence the solution is bounded uniformly for short time; we again use the known bound on $A_\epsilon$ to complete the proof.

It remains only to prove the claim by analyzing (\ref{expressiontwo}). We do this separately for $z''$ inside and outside $r=1/2$, which correspond to off-diagonal and near-diagonal regimes respectively. First suppose that $z''$ is inside $r=1/2$; remember that $z'$ is between $r=3/4$ and $r=5/4$, so $H^{\epsilon Z}(s,z'',z')$ is an off-diagonal heat kernel on $\epsilon Z$. Therefore, we may proceed precisely as in the proof of Lemma \ref{errorlemma}, with $(z',z'')$ instead of $(z,z')$, to conclude that for each $M$ there exists $C_M$ such that
\[|H^{\epsilon Z}(s,z'',z')|<C_Mt^M(W_\epsilon(z''))^{-a}\]
for all $z''$ inside $r=1/2$, all $z'$ between $r=3/4$ and $5/4$, all $\epsilon<1/2$, and all $s<T$. Plugging this bound into (\ref{expressiontwo}) gives a bound for (\ref{expressiontwo}) of
\begin{equation}\label{itsfinite}\int_{\Omega_\epsilon}\chi_1(z'')(W_\epsilon(z''))^{-2a}.\end{equation}
Since $a<n/2$, it is easy to compute directly that (\ref{itsfinite}) is bounded, uniformly in $\epsilon$, which completes the proof of the claim for $z''$ inside $r=1/2$.

On the other hand, suppose that $z''$ is outside $r=1/2$; in this region, $W_\epsilon(z'')$ is between $1/2$ and $3/2$ and hence we may ignore the factor of $W_\epsilon(z'')^a$. We rescale the heat kernel and switch to an integral over $Z$; let $\chi_3$ be equal to $\tilde\chi_1$ outside $r=1/2$ and $0$ inside $r=1/2$. Letting $\sigma=s/\epsilon^2$, $\hat z=z''/\epsilon$, $\tilde z=z'/\epsilon$, $\hat x=|\hat z|^{-1}$ and $\tilde x=|\tilde z|^{-1}$, we get
\begin{equation}\label{rewrittenexptwo}\big |\int_N\int_{\epsilon/2}^{3\epsilon/2} \chi_3(\epsilon/\hat x)H^Z(\sigma,\hat x,y'',\tilde x,y')\chi_1(\epsilon/\tilde x)\ \hat x^{-n-1}\ d\hat x\ dy''\big |.
\end{equation}
The support conditions imply that $\tilde x$ is between $\epsilon/2$ and $3\epsilon/2$, while $\tilde x$ is between $4\epsilon/5$ and $4\epsilon/3$. So the ratio $\hat x/\tilde x$ is between $1/3$ and $2$. Let $\chi_4$ be a smooth cutoff, equal to 1 between $1/3$ and $2$ and 0 outside $(1/4,3)$. The expression (\ref{rewrittenexptwo}) is therefore bounded, independent of $\epsilon$, by \begin{equation}\label{newexpr}\int_N\int_{0}^{\infty} \chi_4(\hat x/\tilde x)\big |H^Z(\sigma,\hat x,y'',\tilde x,y')\big | \hat x^{-n-1}\ d\hat x\ dy''.
\end{equation}

Consider the integrand in (\ref{newexpr}). It is phg conormal on the space $Z^2_{heat}(\sigma,\hat x,y'',\tilde x,y')$ in Figure \ref{hkac}, but it is also supported near the diagonal; that is, away from lb, rb, lb$_0$, and rb$_0$. We want to apply the pushforward theorem of Melrose \cite{me3} (see also Theorem 3.10 in \cite{gri}) to understand this integral. To this end, let $X$ be the manifold with corners given by taking the face rb in Figure \ref{hkac} and crossing it with $N$. We claim that the map from $Z^2_{heat}(\sigma,\hat x,y'',\tilde x,y')$ to $X$ given by $(\sigma,\hat x,y'',\tilde x,y')\rightarrow (\sigma,\tilde x,y')$ is a b-fibration. To see this, note that $X$ has four boundary faces. We label them $H_1$,\ldots,$H_4$, where $H_1$ is rb $\cap$ rb$_0$, $H_2$ is rb $\cap$ bf$_0$, $H_3$ is rb $\cap$ sc, and $H_4$ is rb $\cap$ $\{\sqrt{\sigma}=0\}$. Then all the faces of $Z^2_{heat}$ with $\sigma=0$ are mapped into $H_4$. The faces lb, bf, and sc are mapped into $H_3$. As for $H_2$, its preimage is the union of bf$_0$ and lb$_0$. Finally, zf and rb$_0$ are mapped into $H_1$. We see that the image of each boundary hypersurface of $Z^2_{heat}$ is a boundary hypersurface of $X$, and it is clear from Figure \ref{hkac} that the projection map is a fibration over the interior of each boundary face of $Z^2_{heat}$. Therefore, from \cite[Def. 3.9]{gri}, the map is a b-fibration; moreover, integration in $\hat x$ and $y''$ is precisely pushforward by this map. By Melrose's pushforward theorem, (\ref{newexpr}) is phg conormal on $X$. The absolute value is irrelevant, even if the heat kernel is not assumed to be positive\footnote{We do not want to assume the heat kernel is positive, in part because we are interested in extending the argument to the heat kernel on forms, where the absolute value would be replaced by a norm. In this case, the absolute value or norm of $H^Z$ might not be phg conormal, due to the absolute value causing a lack of smoothness in the coefficients. However, we are only interested in a bound, so we can replace $|H^Z|$ by a strictly larger `smoothed-out' function which is phg conormal with the same orders as $H^Z$ at each face where $H^Z$ does not decay to infinite order, and with very high leading order (i.e. 100) at the faces where $H^Z$ does decay to infinite order. To find such a function, pick any function $g$ which is phg conormal on $Z^2_{heat}$, positive in the interior, and has the same orders as $H^Z$ at all boundary hypersurfaces, except that it has very high leading order at each boundary hypersurface where $H^Z$ has infinite order. Then $H^Zg^{-1}$ has non-negative order at each boundary hypersurface and hence is bounded, so there is a constant $C$ such that $|H^Z|<Cg$. Then we instead apply the pushforward theorem to $Cg$. The argument still works with very high leading orders rather than infinite leading orders at lb, rb, bf, et cetera.}, because we are only interested in the leading orders of (\ref{newexpr}) at the various boundary hypersurfaces of $X$. It remains to compute those orders.

- First consider $H_4$. Its preimage is the union of the face sca and the off-diagonal face at $\sigma=0$. On the non-diagonal face, the integrand of (\ref{newexpr}) has infinite order decay. However, the integrand has order $-n$ at sca, as a function. Since we are integrating with respect to $dz=d(z-z'')$ instead of $d((z-z'')/\sqrt \sigma)$ (which is the natural factor of integration near sca), we get an order shift of $n$ when applying the pushforward theorem. The result is that when applying the pushforward theorem we get a leading order of 0 at $H_4$. (This order shift can be understood by considering the heat kernel on a compact manifold - the integral of the heat kernel in one of the spatial variables converges to 1 even though the heat kernel itself blows up to order $n/2$ at the $t=0$ diagonal).

- As for $H_3$, lb, bf and sc are mapped into it. The integrand is bounded away from lb, has infinite-order decay at bf, and has order 0 with respect to scattering half-densities at sc, which are the appropriate half-densities to use there; the pushforward theorem gives a leading order of 0.

- At $H_2$, the faces in the preimage are bf$_0$ and lb$_0$. At bf$_0$, the integrand has leading order $n$ with respect to scattering half-densities, hence leading order $0$ with respect to the appropriate b-half-densities; since the integrand vanishes to infinite order at lb$_0$, the pushforward theorem again gives a leading order of 0.

- The faces zf and rb$_0$ are mapped into $H_1$. The integrand is supported away from rb$_0$ and has leading order at worst 0 at zf, so the pushforward leading order is 0.
\footnote{A key point in all these computations is that the integrand fails to decay to very high order at only one of the boundary hypersurfaces in the preimage of each hypersurface of rb$\times N_y$. The result is that there are no extended unions in the pushforward theorem and hence no logarithmic blow-up.}

We conclude that (\ref{newexpr}) has non-negative leading orders at each boundary hypersurface of $X$. It is therefore uniformly bounded, which proves the claim and completes the proof of the lemma.
\end{proof}

We can apply the lemma with $A_{\epsilon}=K_{\epsilon}$ or 
$A_{\epsilon}=E_{\epsilon}$, so we conclude that
all convolutions involving $G$ and either $K$ or $E$
are well-defined. This argument completes the proof of Proposition \ref{hkform}.

\subsection{Proof of Theorem \ref{structure}}

Now that we have shown Proposition \ref{hkform}, we consider the trace of the heat kernel,
which may be written as a sum of three terms as in (\ref{threeterms}). 
We now go term-by-term and prove that each is phg conormal on $Q_{0}$.
In all cases, we restrict so that $t$ is less than a fixed time $T$.
First we show:
\begin{lemma}\label{paramphgc}
 $Tr G_{\epsilon}(t)$ is phg conormal on $Q_{0}$. Its leading orders
on $Q_{0}$ are at worst $-n$ at L, $-n$ at F, and 0 at R. \end{lemma}

\begin{proof}
Tr $G_{\epsilon}(t)$ has two terms. The trace of the $\Omega_{0}$ term
is just \begin{equation}\label{omeganoughtpart}
\int_{\Omega_{0}}\chi_{2}(z)H^{\Omega_{0}}(t,z,z)\ dz,\end{equation} which is
independent of $\epsilon$. Since the support of $\chi_{2}(z)$ is away from the
conic tip, the integrand has a phg conormal
expansion as $t\rightarrow 0$ with
smooth coefficients in $z$ by Proposition \ref{conehkstructure}.
Therefore, (\ref{omeganoughtpart}) is phg conormal as $t\rightarrow 0$ and
independent of $\epsilon$, hence phg conormal on $Q$ and therefore on $Q_{0}$.
The leading orders of (\ref{omeganoughtpart}) on $Q$ are $-n$ at $\sqrt t=0$
and $0$ at $\epsilon=0$, so the leading orders on $Q_{0}$ are $-n$ at L and F,
and 0 at R, precisely as needed.

For the $\epsilon Z$ term, we rescale, then change variables. Writing
$\tau=t/\epsilon^{2}$, we get:
\[\int_{\epsilon Z}\chi_{1}(z)H^{\epsilon Z}(t,z,z)\ dz
=\epsilon^{-n}\int_{\epsilon Z}\chi_{1}(z)H^{Z}(t/\epsilon^{2},
z/\epsilon,z/\epsilon)\ dz\]
\begin{equation}\label{zpart}
=\int_{Z}\chi_{1}(z\epsilon)H^{Z}(\tau,z,z)\ dz.\end{equation}

We observe that this integral is precisely (\ref{rememberme}). Recall from section 3 and \cite{s1} that (\ref{zpart}) is phg in $(\tau,\epsilon)$ for $\tau$ bounded above and on $X_b^2(\epsilon,\tau^{-1/2})$ for $\tau$ bounded below, which translates into polyhomogeneity in $(\eta,\sqrt t)$ as in section 3. Therefore
(\ref{zpart}) is phg conormal on $Q_{0}$.
It remains only to analyze the leading orders. This may be done by using
the pushforward theorem and our knowledge of the leading orders of the heat kernel on $Z$ to track the leading orders through the proof of the polyhomogeneity of (\ref{zpart}) (the proof of Theorem 12 in \cite{s1}).
For $\tau$ less than any fixed time $T$, we do not need to use the pushforward
theorem; instead, Theorem \ref{maincor} implies
$H^{Z}(\tau,z,z)\leq C\tau^{-n}$, and hence that
\begin{equation}
\int_{Z}\chi_{1}(z\epsilon)H^{Z}(\tau,z,z)\ dz\leq C\tau^{-n}Vol(Z\cap\{|z|
\leq 2/\epsilon\})=\hat C\tau^{-n}\epsilon^{-n}.\end{equation}
Therefore the leading orders of (\ref{zpart}) at L and F are each at worst
$-n$.

Now consider $\tau$ greater than any fixed time $T$; we need to 
analyze the leading orders at R.
The portion of the integral (\ref{zpart}) with $|z|\leq 1$ is a function
only of $\tau$ (as the cutoff function is identically 1 in this region
for $\epsilon<1/2$), and hence has leading order 0 at R. We therefore
restrict attention to the $|z|\geq 1$ part of (\ref{zpart}), which,
as in the proof in \cite{s1}, may be written:
\begin{equation}\label{umpteen}
\int_{N}\int_{0}^{1}\chi_{1}(\epsilon/x)F(\eta,x,y,x,y)
x^{-n}\ \frac{dx}{x}\ dy.\end{equation}

Here we must use the pushforward theorem. First note that the map from $X_b^{3}(x,\eta,\epsilon)\times N$ to $X_b^2(x,\eta)\times N$ induced by projection off $\epsilon$ is a b-fibration \cite{ms}, as are the maps induced by projection off $x$ or $\eta$ respectively. The function $F(\eta,x,y,x,y)$ is phg conormal on $X_b^2(\eta,x)\times N$, with leading order 0 at $x=0$ (i.e. sc) and leading orders which we call $\kappa_1$ at $\eta=0$ (i.e. zf) and $\kappa_2$ at $\eta=x=0$ (i.e. bf$_0$). The rest of the integrand, $x^{-n}\chi_1(\epsilon/x)$, is phg conormal on $X_b^2(x,\epsilon)$, with leading orders $-n$ at $x=\epsilon=0$, $\infty$ at $x=0$ (because of the cutoff), and $0$ at $\epsilon=0$. Using Melrose's pullback theorem \cite{me3,gri}, the integrand in (\ref{umpteen}) is phg conormal on $X_{b}^{3}(x,\eta,\epsilon)\times N$, with leading orders (with respect to $\frac{dx}{x})$ of $\infty$ at the $x=0$ face and at the $x=\eta=0$ face, $\kappa_2-n$ at the 
$x=\eta=\epsilon=0$ face, 0 at the $\epsilon=0$ face, 
$-n$ at the $x=\epsilon=0$ face, and $\kappa_1$ at the $\eta=0$ and $\eta=\epsilon=0$
faces. Applying the pushforward theorem shows that (\ref{umpteen})
is phg conormal on $X_{b}^{2}(\eta,\epsilon)$, with leading orders of
$-n$ at the $\epsilon=0$ face, $\kappa_2-n\bar{\cup}\kappa_1$ at the $\epsilon=\eta=0$ face, and
$\kappa_1$ at the $\eta=0$ face; here the operation $\bar{\cup}$ is the extended union, which is discussed in \cite{me,me2,gri}.

We may simply apply this analysis with the known values of $\kappa_2=n$ and $\kappa_1=0$, but then we obtain a leading order of $\log\epsilon$ at $\epsilon=\eta=0$ from the extended union; as we will see, that face corresponds to R. To avoid this difficulty, we apply the remainder of the alternative hypothesis. Since the zero-order coefficient of $F$ at zf decays to order strictly greater than $n$ at bf$_0$, we may write $F=F_1+F_2$, where $F_1$ decays to order $0$ at zf and $n+\delta_1$ for some $\delta_1>0$ at bf$_0$, and $F_2$ decays to order $\delta_2>0$ at zf and $n$ at bf$_0$. Applying the above analysis to $F_1$ and $F_2$ separately, the extended union problem vanishes, and we are left with a leading order of 0 at $\epsilon=\eta=0$. As we claimed, this face corresponds to R; since $\epsilon/\eta=\sqrt t$ is bounded, (\ref{umpteen}) has a phg conormal expansion in $(\eta,\epsilon/\eta)=(\eta,\sqrt t)$, with leading orders $0$ at $\eta=0$ and $-n$ at $\sqrt t=0$. Therefore the leading order of (\ref{zpart}) at R is 0.
This completes the proof of Lemma \ref{paramphgc}.
\end{proof}

\begin{remark} If we also use the pushforward theorem to analyze (\ref{umpteen}) for small $\tau$, we obtain a useful consequence which was needed in Section \ref{epsilonzee}. For small $\tau$, we know from \cite{s1} that $H^Z(\tau,x,y,x,y)$ is phg conormal on $[0,1]_{\sqrt\tau}\times[0,1]_x$ with index set precisely $-n+2\mathbb N_0$ at $\sqrt\tau=0$ (corresponding to the usual interior heat asymptotics). The integral (\ref{umpteen}) becomes
\[\int_Z\chi_1(\epsilon/x) H^Z(\tau,x,y,x,y)\ x^{-n-1}\ dx\ dy.\]
Using the pullback theorem as above, the integrand is phg conormal on $X_b^2(x,\epsilon)\times[0,1]_{\sqrt\tau}\times N$, and moreover the index set at the $\sqrt\tau=0$ face is precisely $-n+2\mathbb N_0$. The map from $X_b^2(x,\epsilon)\times[0,1]_{\sqrt\tau}\times N$ to $[0,1]_{\epsilon}\times[0,1]_{\sqrt\tau}$ given by projection off $x$ and $y$ is a b-fibration, and so the integral (\ref{umpteen}) is phg conormal in $(\epsilon,\sqrt\tau)$. The only face which the b-fibration maps into $\sqrt\tau=0$ is the original $\sqrt\tau=0$ face, and hence the index set of (\ref{umpteen}) at $\sqrt\tau=0$ is $-n+2\mathbb N_0$. Since the renormalized heat trace is the finite part of (\ref{umpteen}) at $\epsilon=0$, we conclude that, as claimed at the end of Section \ref{epsilonzee},
\begin{proposition}\label{indexsetatzero} The index set of $^RTr H^Z(\tau)$ at $\sqrt\tau=0$ is $-n+2\mathbb N_0$. \end{proposition}
\end{remark}

Returning to the matter at hand, we now consider $(E_{\epsilon}*G_{\epsilon})(t,z,z')$; note that it 
is zero for $z$ outside the cutoff region - that is, the union
of the bands where $V_{1}$ and $V_{2}$ are supported. We will prove
the following:
\begin{lemma}\label{technicallemma} [Technical Lemma]
For $z$ and $z'$ in the support of $V_{1}$ and $V_{2}$,
$(E_{\epsilon}*G_{\epsilon})(t,z,z')$
is phg conormal on $Q_{0}$, with smooth dependence on $z$ and $z'$.
\end{lemma}
The proof of this lemma is quite long, and so it is deferred until the next section; we assume it for the moment. Once we have proven the lemma, though, it is easy to refine it further:
\begin{corollary}\label{techcor}
For $z$ and $z'$ in the support of $V_{1}$ and $V_{2}$,
$(E_{\epsilon}*G_{\epsilon})(t,z,z')$ is actually phg conormal on $Q$, with infinite-order decay at $t=0$ and leading order at worst 0 at $\epsilon=0$.
\end{corollary}
\begin{proof} We obtain this result as an immediate consequence of Lemma \ref{technicallemma} and Lemma \ref{decay}. The bounds in Lemma \ref{decay} imply that $(E_\epsilon*G_\epsilon)(t,z,z')$ has infinite-order decay at L and F and leading order at worst 0 at R. In turn, since $(E_\epsilon*G_\epsilon)(t,z,z')$ has infinite-order decay at the blown-up face in $Q_0$, a standard argument implies that it is phg conormal on the blown-down space $Q$; see Proposition A.3 in \cite{s}.
\end{proof}

Now consider $(K_{\epsilon}*(E_{\epsilon}*G_{\epsilon}))(t,z,z')$; it is well-defined by associativity and Lemma \ref{decay}. By Corollary \ref{techcor} and Lemma \ref{kerrorlemma}, both $K_\epsilon$ and $E_\epsilon*G_\epsilon$ are phg conormal on $Q$, with infinite-order decay at $t=0$ and leading order 0 at $\epsilon=0$, for $z$ and $z'$ in the cutoff region. We may therefore apply the same argument as in the proof of Lemma \ref{kerrorlemma} to conclude that $K_\epsilon*E_\epsilon*G_\epsilon$ has the same properties. Finally, take the trace of both $E_\epsilon*G_\epsilon$ and $K_\epsilon*(E_\epsilon*G_\epsilon)$; restricting to the diagonal and integrating over the compact cutoff region, we have

\begin{corollary}$Tr (E_{\epsilon}*G_{\epsilon})(t)$ and $Tr (K_{\epsilon}*(E_{\epsilon}*G_{\epsilon}))(t)$ are phg conormal on $Q$ with infinite-order decay at $\sqrt t=0$
and leading order 0 at $\epsilon=0$.
\end{corollary}

We have now shown that each of the three terms in (\ref{threeterms}) is phg
conormal on $Q_{0}$. Therefore, their sum $Tr H^{\Omega_{\epsilon}}(t)$ is phg
conormal on $Q_{0}$. Moreover, $Tr H^{\Omega_\epsilon}(t)$ is $Tr G^{\Omega_\epsilon}(t)$ plus a term which is phg conormal on $Q$ with infinite-order decay at $t=0$. This completes the proof of Theorem \ref{structure}, under the weaker assumption of the `alternate hypothesis.'

\appendix

\section{Proof of the Technical Lemma}

In this section, we complete the proof of Theorem
\ref{structure} by proving Lemma \ref{technicallemma}.
We decompose the convolution $(E_{\epsilon}*G_{\epsilon})(t,z,z')$ 
into four terms:
\[\int_{0}^{t}\int_{\epsilon Z}V_{1}(z)H^{\epsilon Z}(t-s,z,z'')
\chi_{1}(z'')H^{\epsilon Z}(s,z'',z')\chi_{1}(z')\ dz''\ ds\]
\[+\int_{0}^{t}\int_{\Omega_{0}}V_{2}(z)H^{\Omega_{0}}(t-s,z,z'')
\chi_{2}(z'')H^{\Omega_{0}}(s,z'',z')\chi_{2}(z')\ dz''\ ds\]
\[+\int_{0}^{t}\int_{\Omega_{0}}V_{1}(z)H^{\epsilon Z}(t-s,z,z'')
(\chi_{1}\tilde\chi_{2})(z'')H^{\Omega_{0}}(s,z'',z')\chi_{2}(z')\ dz''\ ds\]
\[+\int_{0}^{t}\int_{\Omega_{0}}V_{2}(z)H^{\Omega_{0}}(t-s,z,z'')
(\chi_{2}\tilde\chi_{1})(z'')H^{\epsilon Z}(s,z'',z')\chi_{1}(z')\ dz''\ ds.\]
The spatial integrals are all over $\Omega_{\epsilon}$, but the first
integrand is zero when $|z''|>3/2$ and hence the integral may be written
as an integral over $\epsilon Z$; the remaining integrands are zero when 
$|z''|<1/2$ and hence their integrals may be written as integrals over $\Omega_{0}$.
We will examine each integral in turn and prove that it is phg conormal on
$Q_{0}$, uniformly for $z$ and $z'$ in the cutoff region. These proofs
depend heavily on geometric microlocal analysis and Melrose's pushforward theorem. Some notes on the analysis to follow:

- As before, we may treat $V_{1}(z)$ and $V_{2}(z)$ as cutoff functions,
rather than differential operators; they lift to b-derivatives and thus
have no effect on the polyhomogenous expansions. We suppress these
cutoff functions in the integrals, as they do not interact with the variables
of integration.

- The support conditions imply that $d_{\Omega_{\epsilon}}(z,z'')>1/16$
whenever the integrand is nonzero, so
$z$ and $z''$ are uniformly separated. On the other hand, $z'$ may be close
to either $z$ or $z''$.

- Since we are only interested in polyhomogeneity, we suppress all
polynomial factors of the variables of integration; they may affect the leading
orders, but not polyhomogeneity. All of the overall integrals, as well as
all of the interior spatial integrals, are well-defined; as long as we
avoid integrating the kernel on the diagonal in time down to $t=0$, we
will not have to worry about the integrability condition in the
pushforward theorem.

- Let $\rho=s/\epsilon^{2}$, $\sigma=\sqrt s/(\eta\sqrt t)$, and 
$\sigma'=(\sigma)^{-1}=(\eta\sqrt t)/\sqrt s$. It is often convenient
to integrate in one of these variables instead of in $s$; the change of variables
only gives polynomial factors, which we ignore.

- Many of the integrals will have to be broken into further pieces;
among other decompositions,
we often have to distinguish between the situations when $\tau$ is
bounded above and when it is bounded below.

\subsection{The first integral}
For the first integral, we rescale the heat kernels,
then substitute $\tilde z=z''/\epsilon$ so that we can
integrate over $Z$ instead of $\epsilon Z$. We get
\begin{equation}
\epsilon^{-n}\int_{0}^{t}\int_{Z}H^{Z}(\frac{t-s}
{\epsilon^{2}},z/\epsilon,\tilde z)
\chi_{1}(\epsilon\tilde z)H^{Z}(\frac{s}{\epsilon^{2}},\tilde z,z'/\epsilon)
\ d\tilde z\ ds.\end{equation}
Using $x$ and $y$ instead gives, where $\tilde x=\epsilon x''$:
\begin{equation}\label{firstint}
\epsilon^{-n}\int_{0}^{t}\int_{Z}H^{Z}(\frac{t-s}{\epsilon^{2}},
\epsilon x,y,\tilde x,y'')\chi_{1}(\epsilon/\tilde x)
H^{Z}(\frac{s}{\epsilon^{2}},\tilde x,y'',\epsilon x',y')
\tilde x^{-n-1} d\tilde x\ dy''\ ds.\end{equation}

\subsubsection{Small-$\tau$ case:}

First consider (\ref{firstint}) where $\tau<C$ for some $C$; 
in this case, since $s\leq t$, we use $\rho=s/\epsilon^{2}$, and $\rho<C$ also.
Both heat kernels are thus in the short-time case.
Rescaling the integral in $s$ to an integral in
$\rho$ and suppressing polynomial factors, we obtain
\begin{equation}\label{firstintone}
\int_{0}^{C}
\chi(\{\rho\leq\tau\})
\int_{Z}H^{Z}(\tau-\rho,\epsilon x,y,\tilde x,y'')
\chi_{1}(\frac{\epsilon}{\tilde x})
H^{Z}(\rho,\tilde x,y'',\epsilon x',y')\ d\tilde x\ dy''\ d\rho.\end{equation}


The first heat kernel is at short time and away from the diagonal,
so by Theorem \ref{zshorttime},
it is phg conormal on $X_{b}^{2}(\epsilon x,\tilde x)\times [0,C]_{(\tau-\rho)}
\times N_{y}\times N_{y''}$; $\chi_{1}(\epsilon/\tilde x)$ is also phg conormal
on $X_{b}^{2}(\epsilon x,\tilde x)$, so its pullback is phg conormal on the
same space.
Since $x$ and $x'$ vary over subsets of $[1/2,2]$,
the first heat kernel and the cutoff are also phg conormal
on $X_{b}^{2}(\epsilon x',\tilde x)\times
[0,C]_{(\tau-\rho)}\times N_{y}\times N_{y''}$ with parametric dependence on $x$.
The second heat kernel is, again by Theorem
\ref{zshorttime},
phg conormal on $[Z_{sc}^{2}(\tilde x,y'',\epsilon x',y')\times [0,C]_{\rho};
\{\sqrt\rho=0,\tilde x=\epsilon x',y''=y'\}]$. 
Therefore, the pullback theorem implies that the integrand in (\ref{firstintone})
is phg conormal on
\begin{equation}\label{spaceone}
[Z_{sc}^{2}(\tilde x,y'',\epsilon x',y')\times [0,\sqrt C]_{\sqrt\rho};
\{\sqrt\rho=0,\tilde x=\epsilon x',y''=y'\}]
\times [0,C]_{(\tau-\rho)},\end{equation}
with parametric dependence on $(x,y)$.
Integration in $\tilde x$ and $y''$ 
is a $b$-fibration from (\ref{spaceone})
onto $[0,1]_{\epsilon x'} \times N_{y'}\times [0,1]_{\sqrt\rho}
\times [0,1]_{\tau-\rho}$ (see propositions A.11 and A.12 in \cite{s}, originally from \cite{ms} and \cite{me2}). Therefore, we may apply the pushforward theorem;
the spatial integral in (\ref{firstintone})
is phg conormal on $[0,1]_{\epsilon x'} \times N_{y'}\times 
[0,C]_{\sqrt\rho}\times [0,C]_{(\tau-\rho)}$ with $(x,y)$ as parameters, 
and hence phg conormal 
on $[0,1]_{\epsilon}\times[0,C]_{\rho}\times [0,C]_{(\tau-\rho)}$
with $(x,y,x',y')$ as parameters. 

All dependence on $(x,y,x',y')$ is smooth, so we now suppress it.
It remains to do the $\rho$ integral. Notice that the map from
$[0,C]_{\tau}\times [0,1]_{\rho/\tau}$ to $[0,C]_{\rho}\times [0,C]_{\tau-\rho}$
given by $(x_{1},x_{2})\rightarrow (x_{1}x_{2},x_{1}(1-x_{2}))$ is a b-map.
Therefore, the spatial integral above is phg conormal in $[0,1]_{\epsilon}
\times[0,C]_{\tau}\times[0,1]_{\rho/\tau}$.
Change variables to an integral in $\rho/\tau$ and integrate from 0 to 1;
this integration is a b-fibration onto $[0,1]_{\epsilon}\times [0,C]_{\tau}$,
so 
(\ref{firstintone}) is phg conormal
in $(\epsilon,\tau)$ and hence on $Q_{0}$.


\subsubsection{Large-$\tau$ case, introduction:}

Assume that $\tau>C$ for some large $C$, and use $(\eta,t)$ in place
of $(\tau,\epsilon)$.
We break up the $s$ integral into four pieces: 
I, from 0 to $\epsilon^{2}=\eta^{2}t$; 
II, from $\eta^{2}t$ to $t/2$, IV, from $t/2$ to $t-\epsilon^{2}$, and
III, from $t-\epsilon^{2}$ to $t$. Each is 
a different regime.

In regions I and II, the time argument of the first heat kernel
is greater than $\tau/2>C/2$, so we write 
it as $F(\frac{\eta\sqrt t}{\sqrt{t-s}},\eta\sqrt t x,y,\tilde x,y'')$,
which equals
$F(\eta\sqrt{\frac{1}{1-s/t}},\eta\sqrt t x,y,\tilde x,y'')$.
We introduce a placeholder variable $a=\eta\sqrt t$.
The total integral in regions I and II is, again suppressing polynomials:
\[\int_{0}^{t/2}\int_{Z}
F(\eta\sqrt{\frac{1}{1-(s/t)}},a x,y,\tilde x,y'')
\chi_{1}(a/\tilde x)H^{Z}(\frac{s}{\eta^{2}t},\tilde x,y'',
ax',y')\ d\tilde x\ dy''\ ds.\]

The support conditions guarantee that the arguments of $F$ are away from sc.
Therefore, $F$ is phg conormal
on $X_{b}^{3}(\eta\sqrt{\frac{t}{t-s}},a,\tilde x)$ with
smooth parametric dependence on $(x,y,y'')$. However, $\sqrt{t/(t-s)}$ is a
smooth function of $s/t$, and is bounded away from zero since $s\leq t/2$ in regions I and II.
So: in regions I and II,
$F$ is actually phg conormal on $X_{b}^{3}(\eta,a,\tilde x)$,
with smooth parametric dependence on $(x,y,y'')$ and also on $s/t$, which
is between 0 and 1/2.

\subsubsection{Region I:}
In region I, we have
\begin{equation}\label{intregionone}\int_{0}^{\eta^{2}t}\int_{Z}
F(\eta\sqrt{\frac{1}{1-(s/t)}},a x,y,\tilde x,y'')
\chi_{1}(a/\tilde x)H^{Z}(\frac{s}{\eta^{2}t},\tilde x,y'',
ax',y')\ d\tilde x\ dy''\ ds.\end{equation}
The temporal integral is from $s=0$ to $s=\eta^{2}t$, so the 
second term is a short-time heat kernel. We break the integral into
two pieces, called I-a and I-b, 
with $\tilde x<2\eta\sqrt t$ and $\tilde x>2\eta\sqrt t$ respectively.

Consider I-a; here $\tilde x/(\eta\sqrt t)$ is between
$1/2$ and 2. The first heat kernel $F$ is phg conormal on $X_{b}^{3}(\eta,
a,\tilde x)$, as is the cutoff $\chi_{1}(a/\tilde x)$. 
Since $\tilde x/a$ is between 1/2 and 2, $F$ is in fact
phg conormal on $X_{b}^{2}(\eta,a)$ with parametric dependence on $(\tilde x/
(\eta\sqrt t))$ and $s/t$ as well as on $(x,y,y'')$. Using $\sigma=\sqrt s/(\eta\sqrt t)$,
the second term is phg conormal on the space
\[
[Z_{sc}^{2}(\tilde x,y'',\eta\sqrt t x',y')\times[0,1]_{\sigma};
\{\sigma=0,\tilde x=\epsilon x',y''=y'\}].\]
Combining all the terms, the integrand is phg conormal on:
\begin{equation}\label{spacethree}X_{b}^{2}(a,\eta)\times(
[Z_{sc}^{2}(\tilde x,y'',\eta\sqrt t x',y')\times[0,1]_{\sigma};
\{\sigma=0,\tilde x=\epsilon x',y''=y'\}]),\end{equation}
with parametric dependence on $s/t=(\eta\sigma)^{2}$ as well
as $(x,y,y'')$. The spatial integral is in $\tilde x$ and $y''$, and 
the pushforward map is a b-fibration from (\ref{spacethree}) onto
$X_{b}^{2}(a,\eta)
\times [0,1]_{\eta\sqrt t x'}\times [0,1]_{\sigma}$; its fibers
are transverse to the cutoff at $\tilde x=2\eta\sqrt t$.
By the pushforward theorem, the part of (\ref{intregionone}) with
$\tilde x<2\eta\sqrt t$ is phg conormal on
$X_{b}^{2}(a,\eta)\times [0,1]_{\eta\sqrt t}\times [0,1]_{\sigma}$ 
with parametric dependence on $(x,y,x',y')$ and also on $\eta^{2}\sigma^{2}$.

Now we consider the temporal integral; it goes from $0$ to $1$ in $\sigma$.
We first need to restrict to $a=\eta\sqrt t$ in the space $X_{b}^{2}(a,\eta)
\times [0,1]_{\eta\sqrt t}\times [0,1]_{\sigma}$. 
$X_{b}^{3}(a,\eta,\eta\sqrt t)\times [0,1]_{\sigma}$ is a
blow-up of this space, and restricting to $a=\eta\sqrt t$ gives a function
that is phg conormal on $X_{b}^{2}(\eta,\eta\sqrt t)\times [0,1]_{\sigma}$
with parametric dependence on $(\eta\sigma)^{2}$. However, this parametric
dependence amounts to multiplying each coefficient by a smooth function of
$(\eta\sigma)^{2}$; we conclude that the spatial integral in
(\ref{intregionone}) is phg conormal on $X_{b}^{2}(\eta,\eta\sqrt t)
\times [0,1]_{\sigma}$. Integration in $\sigma$ is then a b-fibration
onto $X_{b}^{2}(\eta,\eta\sqrt t)$. So the part of (\ref{intregionone}) 
with $\tilde x<2\eta\sqrt t$ is phg conormal on $X_{b}^{2}(\eta,\eta\sqrt t)$.

So far we have treated $\eta\sqrt t$ as a formal variable, but now we
remember that $\eta\sqrt t/\eta=\sqrt t$ is bounded. Therefore we have an
expansion in $(\eta,(\eta\sqrt t)/\eta)=(\eta,\sqrt t)$, and the part
of (\ref{intregionone}) with $\tilde x<2\eta\sqrt t$ is phg conormal in
$(\eta,\sqrt t)$ and hence on $Q_{0}$.

\medskip

Now consider integral I-b. This is easier, as the second term is away from
the diagonal. The first term is phg conormal on 
$X_{b}^{3}(\eta,a,\tilde x)$
with parametric dependence on $(x,y,y'')$ and $\sqrt s/\sqrt t=\sigma\eta$,
and the second term is now
phg conormal on $X_{b}^{2}(\tilde x,\eta\sqrt t)\times [0,1]_{\sigma}$
with parametric dependence on $(x',y',y'')$. The cutoff functions are also phg
conormal there. We immediately restrict to $a=\eta\sqrt t$ and see that
the integrand is phg conormal on
\[X_{b}^{3}(\eta,\eta\sqrt t,\tilde x)\times [0,1]_{\sigma}\] with parametric
dependence on $(x,y,x',y',y'')$ and $\eta\sigma$. As before, the parametric
dependence on $(\eta\sigma)$ does not affect the polyhomogeneity. 
Integration in $\tilde x$ and $y''$ is a
b-fibration onto $X_{b}^{2}(\eta,\eta\sqrt t)\times [0,1]_{\sigma}$
with fibers transverse to the cutoff $\tilde x=2\eta\sqrt t$, so we
can use the pushforward theorem; the analysis then proceeds exactly as for
integral I-a. Therefore (\ref{intregionone}) is phg conormal on $Q_{0}$.

\subsubsection{Region II:}

For region II, the third argument of the heat kernel is $s/(\eta^{2}t)$, which is
now larger than 1. So both heat kernels are long-time. We must analyze:
\begin{equation}\label{intregiontwo}\int_{\eta^{2}t}^{t/2}\int_{Z}
F(\eta\sqrt{\frac{t}{t-s}},\eta\sqrt t x,y,\tilde x,y'')
\chi_{1}(\eta\sqrt t/\tilde x)F(\eta\frac{\sqrt t}{\sqrt s},\tilde x,y'',
\eta\sqrt t x',y')\ d\tilde x\ dy''\ ds.\end{equation}

As before, break the $\tilde x$ integral into two pieces:
II-a, where $\tilde x<2\eta\sqrt t$, and
II-b, where $\tilde x>2\eta\sqrt t$. In each pieces,
switch to an integral in $\sigma'=(\eta\sqrt t/\sqrt s)$.
In both pieces, the first kernel is phg conormal on
$X_{b}^{3}(\eta,\eta\sqrt t,\tilde x)$, but with parametric dependence
on $\eta/\sigma'$; hence the first kernel is phg conormal on
$X_{b}^{4}(\eta,\eta\sqrt t,\tilde x,\sigma')$. It has parametric
dependence on $(x,y,y'')$.

Consider region II-a. As in the analysis of I-a, 
the restriction to $\tilde x<2\eta\sqrt t$
means that $\tilde x/(\eta\sqrt t)$ is between 1/2 and 2. We again introduce the
placeholder $a$, but this time we also introduce a placeholder $b$ for $\sigma'$.
Since $\tilde x/(\eta\sqrt t)$ is bounded, the first term is phg conormal on
$X_{b}^{3}(a,\eta,b)$ with parametric dependence on $\tilde 
x/(\eta\sqrt t)$ as well as $(x,y,y'')$. The cutoff is phg conormal
on $X_{b}^{2}(\tilde x,\eta\sqrt t)$.
The second kernel is phg conormal on $Z^{2}_{\sigma',sc}(\sigma',\tilde x,
y'',\eta\sqrt t x', y')$. Therefore, the whole integrand is phg conormal on
\begin{equation}\label{spacefour}
X_{b}^{3}(a,\eta,b)\times Z^{2}_{\sigma',sc}(\sigma',\tilde x,y'',
\eta\sqrt t x', y'),
\end{equation} with parametric dependence on $(x,y)$. Integration
in $\tilde x$ and $y''$ is a pushforward by a b-fibration onto
$X_{b}^{3}(a,\eta,b)\times X_{b}^{2}(\sigma',\eta\sqrt t)$. So the spatial
integral in (\ref{intregiontwo}) over $\tilde x<2\eta\sqrt t$ is phg conormal on 
$X_{b}^{3}(a,\eta,b)\times X_{b}^{2}(\sigma',\eta\sqrt t)$, 
with parametric dependence on $(x,y,x',y')$. By the pullback theorem and the theory of b-stretched products in \cite{ms} (also see the remarks after proposition A.10 in \cite{s}), it is phg conormal
on $X_{b}^{5}(a,\eta,b,\sigma',\eta\sqrt t)$.

We need to resolve the placeholders; restricting to $a=\eta\sqrt t$ 
and then $b=\sigma'$, the result is phg conormal
on $X_{b}^{3}(\eta,\sigma',\eta\sqrt t)$. We now integrate
in $\sigma'$ from $\eta\sqrt 2$ to $1$; this integral is a pushforward
by a b-fibration onto $X_{b}^{2}(\eta,\eta\sqrt t)$. So the part of
(\ref{intregiontwo}) with $\tilde x<2\eta\sqrt t$ is phg conormal
on $X_{b}^{2}(\eta,\eta\sqrt t)$, and thus phg conormal on $Q_{0}$ as in
the analysis of integral I-a.

\medskip

As for II-b, the first kernel is phg conormal on 
$X_{b}^{4}(\tilde x,\eta,\eta\sqrt t,\sigma')$.
The second kernel is now supported away from the scattering diagonal, and hence
is phg conormal on $X_{b}^{3}(\sigma',\tilde x,\eta\sqrt t)$; and the cutoff
is phg conormal on $X_{b}^{2}(\tilde x,\eta\sqrt t)$. By the pullback theorem,
the integrand is phg conormal on $X_{b}^{4}(\tilde x,\eta,\eta\sqrt t,\sigma')$
with parametric dependence in $(x,y,x',y',y'')$. Integration in $\tilde x$
and $\sigma'$, as before, is pushforward by a b-fibration 
onto $X_{b}^{2}(\eta,\eta\sqrt t)$,
and the analysis proceeds as in II-a. We conclude that (\ref{intregiontwo})
is phg conormal on $Q_{0}$ with smooth dependence on $z$ and $z'$.

\subsubsection{Region III:}

In both region III and region IV, the second
term is a long-time heat kernel. Make the substitution $\bar s=t-s$;
our integral becomes (modulo polynomial factors)
\begin{equation}\label{intthreefour}\int_{0}^{t/2}\int_{Z}
H^{Z}(\frac{\bar s}{\eta^{2}t},\eta\sqrt t x,y,\tilde x,y'')
\chi_{1}(\eta\sqrt t/\tilde x)F(\eta\sqrt{\frac{t}{t-\bar s}}
,\tilde x,y'',
\eta\sqrt t x',y')\ d\tilde x\ dy''\ d\bar s.\end{equation}

\medskip

First consider region III, which is the integral (\ref{intthreefour}) from
$s=0$ to $s=\eta^{2} t$. 
Change variables to $\bar\sigma=\sqrt{\bar s}/(\eta\sqrt t)$; it is integrated
from $0$ to $1$. $H^{Z}$ is a short-time heat kernel away from the diagonal
and hence is phg conormal on $X_{b}^{2}(\eta\sqrt t,\tilde x)
\times [0,1]_{\bar\sigma}$, with parametric
dependence on $(x,y,y'')$; $\chi_{1}(\eta\sqrt t/\tilde x)$ is phg conormal
on the same space. Therefore, the first heat kernel and cutoff are
also phg conormal on $X_{b}^{2}(\eta\sqrt t x',\tilde x)\times [0,1]_{\bar\sigma}$,
with parametric dependence on $(x,x',y,y'')$.

As with the integrals in region I, $F$ is phg conormal
on $Z_{\eta,sc}^{2}(\eta,\tilde x,y'',\eta\sqrt t x',y')$,
with parametric dependence on $\sqrt{\bar s}/\sqrt t=\eta\bar\sigma$ between
0 and $1/2$ (the parameter
isn't involved in the spatial coordinates on $Z_{\eta,sc}^{2}$).
Let $c$ be a placeholder for $\eta\bar\sigma$. By the pullback
theorem, the integrand of (\ref{intthreefour}) in region III is phg
conormal on
\begin{equation}\label{spacefive}
Z_{\eta,sc}^{2}(\eta,\tilde x,y'',\eta\sqrt t x',y')
\times [0,1]_{\bar\sigma}\end{equation} with parametric dependence on $(x,y,x',c)$.
The spatial integral is a pushforward by a b-fibration onto
$X_{b}^{2}(\eta,\eta\sqrt t x')\times [0,1]_{\bar\sigma}$,
which is the same as $X_{b}^{2}(\eta,\eta\sqrt t)\times [0,1]_{\bar\sigma}$, 
with parametric dependence on $(x,y,x,',y',c)$.
So in region III, the spatial integral in (\ref{intthreefour}) is phg conormal
on $X_{b}^{2}(\eta,\eta\sqrt t)\times [0,1]_{\bar\sigma}$ with parametric
dependence on $c=(n\bar\sigma)$ as well as $(x,y,x',y')$. We now follow the
analysis of integral I-a, with $\bar\sigma$ in place of $\sigma$, 
to conclude that the part of (\ref{intthreefour})
corresponding to region III is phg conormal on $Q_{0}$.

\subsubsection{Region IV:}

Finally, we have region IV, which corresponds to $\bar s=\eta^{2}t$ to 
$\bar s=t/2$. 
Both heat kernels are long-time, and we need to analyze:
\begin{equation}\label{intregionfour}\int_{\eta^{2}t}^{t/2}\int_{Z}
F(\eta\frac{\sqrt t}{\sqrt{\bar s}},\eta\sqrt t x,y,\tilde x,y'')
\chi_{1}(\eta\sqrt t/\tilde x)F(\eta\sqrt{\frac{t}{t-\bar s}}
,\tilde x,y'',
\eta\sqrt t x',y')\ d\tilde x\ dy''\ d\bar s.\end{equation}

As before, change variables to 
$\bar\sigma'=\eta\sqrt t/\sqrt{\bar s}$; then 
$\bar\sigma'$ goes from $\eta/\sqrt 2$ to 1.
The first term is an off-diagonal long-time heat kernel, and is thus
phg conormal on $X_{b}^{3}(\bar\sigma',\eta\sqrt t,\tilde x)$ with
parametric dependence on $(x,y,y'')$; so is the cutoff function $\chi_{1}(
\eta\sqrt t/\tilde x)$.
The second term is again phg conormal on $Z_{\eta,sc}^{2}(\eta,\tilde x,y'',
\eta\sqrt t x',y')$ with parametric dependence on 
$\sqrt{\bar s}/\sqrt t=\eta/\bar\sigma'$ smooth down to $\sqrt{\bar s}/\sqrt t=0$.
Introduce three placeholders $a=\eta\sqrt t x'$, $b=\tilde x$, $d=\eta$. Then
the integrand of (\ref{intregionfour}) is phg conormal on
\begin{equation}\label{spacesix}X_{b}^{4}(\bar\sigma',a,b,d)
\times Z^{2}_{\eta,sc}(\eta,\tilde x,y'',\eta\sqrt t x', y')
\end{equation} with parametric dependence on $(x,y,x',y'')$.

All the integrals are absolutely convergent since $\bar\sigma'$ does not go
down to zero (and the spatial integrals are over compact sets). So we can switch the order of integration and integrate in $\bar\sigma'$
first. The $\bar\sigma'$ integral is phg conormal on
$X_{b}^{3}(a,b,d)\times Z^{2}_{\eta,sc}(\eta,\tilde x,y'',\eta\sqrt t x',y')$.
Since $a=\eta\sqrt t x'$, $b=\tilde x$, and $d=\eta$, we would like to
say that the $(a,b,d)$ part is phg conormal on $X_{b}^{3}(\eta\sqrt t x',
\tilde x,\eta)$ and hence on its blowup $Z^{2}_{\eta,sc}(\eta,\tilde x,y'',
\eta\sqrt t x',y')$. We could then conclude that the whole $\bar\sigma'$ integral 
is phg conormal on $Z^{2}_{\eta,sc}(\eta,\tilde x,y'',\eta\sqrt tx',y')$.

To justify this assertion, note that
the $\bar\sigma'$ integral has expansions at all boundary faces and corners of 
$Z^{2}_{\eta,sc}(\eta,\tilde x,y'',\eta\sqrt t x',y')$, with coefficients
in $\mathcal A_{phg}^{\mathcal E}(X_{b}^{3}(a,b,d))$ for some fixed index
family $\mathcal E$ on $X_{b}^{3}$. It may therefore be written as a sum
of terms of the form $u_{i}v_{i}$, where $u_{i}$ is phg conormal on
$Z^{2}_{\eta,sc}(\eta,\tilde x,y'',\eta\sqrt t x',y')$ with index family
$\mathcal F_{i}$ approaching infinity, and $v_{i}$ is phg conormal on
$X_{b}^{3}(a,b,d)$ with fixed index family $\mathcal E$. Replacing the
placeholders, $v_{i}$ is phg conormal on $X_{b}^{3}(\eta,\tilde x,\eta\sqrt t x')$
with fixed index family, and hence is phg conormal on the pullback
$Z^{2}_{\eta,sc}(\eta,\tilde x,y'',\eta\sqrt tx',y')$ with fixed index
family $\mathcal E^{\sharp}$. Therefore $u_{i}v_{i}$ is phg
conormal on $Z^{2}_{\eta,sc}(\eta,\tilde x,y'',\eta\sqrt tx',y')$ with
index family $\mathcal E^{\sharp}+\mathcal F_{i}$, which approaches
infinity as $i$ increases. The sum over $i$ is therefore phg conormal
on $Z^{2}_{\eta,sc}(\eta,\tilde x,y'',\eta\sqrt tx',y')$, with
parametric dependence on $(x,y,x')$.

In conclusion, (\ref{intregionfour}) is the integral over
$\tilde x$ and $y''$ of a function which is phg conormal on $Z^{2}_{\eta,sc}
(\eta,\tilde x,y'',\eta\sqrt tx',y')$. By the pushforward theorem,
(\ref{intregionfour}) is phg conormal on $X_{b}^{2}(\eta\sqrt tx',\eta)$
and hence on $X_{b}^{2}(\eta\sqrt t,\eta)$ with smooth dependence on
$(x,y,x',y')$. By the same analysis as in region I, (\ref{intregionfour})
is thus phg conormal on $Q_{0}$. This completes the analysis of region IV,
and hence of the first integral.

\subsection{The second integral:}
The second integral is
\[\int_{0}^{t}\int_{\Omega_{0}}V_{2}(z)H^{\Omega_{0}}(t-s,z,z'')\chi_{2}(z'')
H^{\Omega_{0}}(s,z'',z')\chi_{2}(z')\ dz''\ ds.\]
This integral is independent of $\epsilon$. Moreover,
by the same arguments as in the proof of Lemma \ref{decay}, it decays to infinite
order in $t$ at $t=0$, uniformly in the spatial variables as they
range over the cutoff region. Therefore, this integral is phg
conormal on $Q$ and hence on $Q_{0}$, uniformly in $z$ and $z'$.

\subsection{The third integral:}
The third integral, up to polynomials in $\epsilon$ and suppressing
$V_{1}(z)$ and $\chi_{2}(z')$, is:
\begin{equation}\label{thirdint}
\int_{0}^{t}\int_{\Omega_{0}}H^{Z}(\frac{t-s}{\epsilon^{2}},\frac{z}{\epsilon},
\frac{z''}{\epsilon})(\chi_{1}\tilde\chi_{2})(z'')H^{\Omega_{0}}(s,z'',z')
\ dz''\ ds.\end{equation}
Note that now $z''$ is supported between $r=1/2$ and $r=2$. We again break
it into the small-$\tau$ and large-$\tau$ cases.

\subsubsection{Small-$\tau$ case:}

For small $\tau$, we let $\rho=s/\epsilon^{2}$ and then let $\mu=1-\rho/\tau$.
Suppressing polynomials, we get:
\begin{equation}\label{thirdintsmall}
\int_{0}^{1}\int_{\Omega_{0}}H^{Z}(\mu,\frac{z}{\epsilon},\frac{z''}{\epsilon})
(\chi_{1}\tilde\chi_{2})(z'')H^{\Omega_{0}}(\epsilon^{2}\tau(1-\mu),z'',z')\
dz''\ d\mu.\end{equation}
The first heat kernel is off-diagonal, and hence is phg conormal
on $[0,1]_{\mu}\times[0,1/2)_{\epsilon}$ 
with parametric dependence on $z$ and $z''$.
Let $a$ be a placeholder for $\epsilon^{2}\tau(1-\mu)$; 
the second term is on-diagonal but away from the conic tip, and hence is phg
conormal in $(z'',z',\sqrt a)$ with a blowup at 
$\{\sqrt a=0,z''=z'\}$. 
The whole integrand is phg conormal on
\begin{equation}\label{spaceseven}
[[0,1]_{\mu}\times[0,1/2)_{\epsilon}\times[0,1]_{\sqrt a}\times(z',z'');
\{\sqrt a=0,z'=z''\}],
\end{equation} 
with parametric dependence on $z$. Integration in $z''$
is a b-fibration onto $[0,1]_{\mu}\times[0,1/2)_{\epsilon}\times [0,1]_{\sqrt a}$,
so the spatial integral in (\ref{thirdintsmall}) is phg conormal on 
$[0,1]_{\mu}\times[0,1/2)_{\epsilon}\times [0,1]_{\sqrt a}$ 
and hence on 
$[0,1]_{\mu}\times[0,1/2)_{\epsilon}\times [0,1]_{a}$, 
with parametric dependence on $(z,z')$.

The map from $[0,1]_{\mu}\times[0,1/2)_{\epsilon}\times[0,\sqrt C]_{\sqrt \tau}$
to $[0,1]_{\mu}\times[0,1/2)_{\epsilon}\times [0,1]_{a}$ given by
$(\mu,\epsilon,\tau)\rightarrow(\mu,\epsilon,\epsilon\sqrt\tau(1-\mu))$ is a
b-map, so by the pullback theorem, the spatial integral in
(\ref{thirdintsmall}) is phg conormal
on $[0,1]_{\mu}\times[0,1/2)_{\epsilon}\times[0,\sqrt C]_{\sqrt \tau}$.
Integration in $\mu$ is then a b-fibration onto $[0,1/2)_{\epsilon}
\times[0,\sqrt C]_{\sqrt\tau}$, and hence (\ref{thirdintsmall}) is phg
conormal in $(\epsilon,\sqrt\tau)$ and therefore on $Q_{0}$.

\subsubsection{Large-$\tau$ case:}

Here we use $(\eta,\sqrt t)$, and switch $\bar s=t-s$. Break the integral into
two integrals, one from $\bar s=0$ to $\bar s=\eta^{2}t$ and the other from 
$\eta^{2}t$ to $t$.

For the first integral, let
$\sigma=\sqrt{\bar s}/(\eta\sqrt t)$, and the integral becomes (up to polynomials):
\begin{equation}\label{thirdintlargeone}
\int_{0}^{1}\int_{\Omega_{0}}H^{Z}(\bar\sigma^{2},\frac{z}{\eta\sqrt t},\frac{z''}
{\eta\sqrt t})
(\chi_{1}\tilde\chi_{2})(z'')
H^{\Omega_{0}}(t(1-\bar\sigma^{2}\eta^{2}),z'',z')\ dz''\ d\bar\sigma.\end{equation}
The first term is off-diagonal and hence is phg conormal in 
$(\bar\sigma,\eta\sqrt t)$, and therefore phg conormal in $(\bar\sigma,\eta,\sqrt t)$,
with parametric dependence on $z'$ and $z''$.
As in the small-$\tau$ case, 
let $a$ be a placeholder for $t(1-\bar\sigma^{2}\eta^{2})$,
and then the integrand is phg conormal on
\begin{equation}\label{spaceeight}
[[0,1]_{\bar\sigma}\times[0,C]_{\eta}\times[0,\sqrt T]_{\sqrt t}
\times[0,1]_{\sqrt a}\times(z',z'');
\{\sqrt a=0,z'=z''\}]
\end{equation}
with parametric dependence on $z$. As before, the spatial integral
in (\ref{thirdintlargeone}) is phg conormal on
$[0,1]_{\bar\sigma}\times [0,C]_{\eta}\times [0,T]_{t}\times [0,1]_{a}$,
with parametric dependence on $(z,z')$. 
Therefore
it is also phg conormal on $[0,1]_{\bar\sigma}\times [0,1/2]_{\eta}\times
X_{b}^{2}(t,a)$. 

Since $1-\bar\sigma^{2}\eta^{2}$
is bounded away from zero for small $\eta$,
the set $\{a=t(1-\bar\sigma^{2}\eta^{2})\}$ is a p-submanifold of 
$[0,1]_{\bar\sigma}\times [0,1/2]_{\eta}\times
X_{b}^{2}(t,a)$. By the usual restriction theorems for phg conormal functions (see \cite{me2} or proposition A.8 in \cite{s}), 
the spatial integral in (\ref{thirdintlargeone}) is phg conormal
on $[0,1]_{\bar\sigma}\times [0,1/2]_{\eta}\times[0,1]_{\sqrt t}$.
Integration in $\bar\sigma$ is a b-fibration onto $(\eta,\sqrt t)$ space,
and therefore (\ref{thirdintlargeone}) is phg conormal in $(\eta,\sqrt t)$
and hence on $Q_{0}$.

\medskip

In the second integral, the first heat kernel is long-time, so we switch to $F$.
Switch to $\bar\sigma'=\eta\sqrt t/\sqrt{\bar s}$, and the integral becomes,
modulo polynomials:
\begin{equation}\label{thirdintlargetwo}
\int_{0}^{1}\chi(\{\eta<\bar\sigma'\})
\int_{\Omega_{0}}F(\bar\sigma',\eta\sqrt t x, y,
\eta\sqrt t x'', y'')H^{\Omega_{0}}(t(1-(\frac{\eta}{\bar\sigma'})^{2})),z'',z')\
dz''\ d\bar\sigma'.\end{equation}
The first heat kernel is off-diagonal, 
hence phg conormal on $X_{b}^{2}(\bar\sigma',\eta\sqrt t)$,
with parametric dependence on $(z,z'')$. Let $a$ be a placeholder for 
$t(1-(\eta/\bar\sigma')^{2})$. Then exactly as in the analysis
of (\ref{thirdintlargeone}), the spatial integral in (\ref{thirdintlargetwo})
is phg conormal on
$X_{b}^{2}(\bar\sigma',\eta\sqrt t)\times [0,1]_{a}$, with parametric dependence on $(z,z')$.

Let $\tilde X$ be the portion of $X_{b}^{3}(\bar\sigma',\eta,\eta\sqrt t)$
where $\bar\sigma'>\eta$ and $(\eta\sqrt t)/\eta<\sqrt T$, 
with a boundary face at $\bar\sigma'=\eta$. Then $\tilde X$ is precisely
$[0,\sqrt T]_{\sqrt t}\times [0,1]_{\eta/\bar\sigma'}\times [0,1]_{\bar\sigma'}$.
The map from $\tilde X$ to $X_{b}^{2}(\bar\sigma',\eta\sqrt t)\times [0,1]_{a}$
given by $(\sqrt t,\eta/\bar\sigma',\bar\sigma')\rightarrow 
(\bar\sigma,(\sqrt t(\eta/\bar\sigma')\bar\sigma'),\sqrt t(1-(\eta/\bar\sigma')^{2}))$
is a b-map. Moreover, the cutoff $\chi(\{\bar\sigma'>\eta\})$ is also phg conormal
on $\tilde X$. Therefore the whole $\bar\sigma'$ integrand is phg conormal
on $\tilde X$ and hence on $X_{b}^{3}(\bar\sigma',\eta,\eta\sqrt t)$.
We now follow the analysis of region II-a of the first integral to conclude
that (\ref{thirdintlargetwo}) is phg conormal on $Q_{0}$. This completes
the analysis of the third integral.

\subsection{The fourth integral:}
The fourth integral, suppressing $V_{2}(z)$
and $\chi_{1}(z')$ and again discarding polynomial factors, is:
\[\int_{0}^{t}\int_{\Omega_{0}}H^{\Omega_{0}}
(t-s,z,z'')(\chi_{2}\tilde\chi_{1})(z'')
H^{Z}(\frac{s}{\epsilon^{2}},\frac{z''}{\epsilon},\frac{z'}{\epsilon})\ dz''\ ds.\]
The $H^{\Omega_{0}}$ term is away from the diagonal;
it simply
decays to infinite order in $(t-s)$, uniformly in $z$ and $z''$. The
calculations are therefore a much-simplified version of those in the
analysis of the first integral.

\subsubsection{Small-$\tau$ case:}

First assume that $\tau$ is bounded; let
$\rho=s/\epsilon^{2}$, and we get, suppressing all cutoffs and polynomials:
\begin{equation}\label{fourthintsmall}
\int_{0}^{\tau}\int_{Z}H^{\Omega_{0}}(\epsilon^{2}(\tau-\rho),z,z'')
H^{Z}(\rho,\tilde x, y'',\epsilon x', y')\ d\tilde x\ dy''\ d\rho.
\end{equation}.
The first heat kernel 
is phg conormal on $[0,1]_{\epsilon}\times [0,1]_{\tau-\rho}$,
with parametric dependence on $x$, $y$, $\tilde x$, and $y''$. 
Therefore, the whole spatial integrand is phg conormal on
$[Z_{sc}^{2}(\tilde x,y'',\epsilon x',y')\times [0,1]_{\rho};
\{\sqrt\rho=0, z''=z'\}]\times [0,1]_{\tau-\rho}$ with
parametric dependence on $z$. By the pushforward
theorem, the spatial integral in
(\ref{fourthintsmall}) is phg conormal on $[0,1]_{\rho}\times[0,1]_{\epsilon x'}
\times[0,1]_{\tau-\rho}$. As with the short-time piece of the first integral,
it is actually phg conormal on $[0,1]_{\tau}\times [0,1]_{\rho/\tau}\times
[0,1]_{\epsilon}$; by the same analysis,
(\ref{fourthintsmall}) is phg conormal on $Q_{0}$.

\subsubsection{Large-$\tau$ case:}

Now assume that $\tau>2$. Break the
$s$-integral into two pieces: $s<\eta^{2}t$ and $s>\eta^{2}t$.
For the $s<\eta^{2}t$ portion, we have:
\begin{equation}\label{fourthintlargeone}
\int_{0}^{1}\int_{Z}H^{\Omega_{0}}(t(1-\sigma^{2}\eta^{2})),z,z'')
H^{Z}(\sigma^{2},\tilde x, y'',\eta\sqrt t x', y')\ d\tilde x\ dy''\ d\sigma.
\end{equation}
Since $\eta<1/2$
and $\sigma\leq 1$, the first kernel is 
a smooth function of $(t,\sigma,\eta)$, and hence
smooth on $X_{b}^{2}(\eta\sqrt t,\eta)\times [0,1]_{\sigma}$ (remember that
$\sqrt t=(\eta\sqrt t/\eta)$ is bounded). As in part I-a,
introduce the placeholder $a$ for $\eta\sqrt t$.
The second kernel is phg conormal on the same space as in the small-$\tau$
case above, with $\eta\sqrt t$ replacing $\epsilon$. Therefore,
the whole spatial integrand is phg conormal on
\begin{equation}\label{spacenine}
[Z_{sc}^{2}(\tilde x,y'',\eta\sqrt t x',y')\times [0,1]_{\sigma};
\{\sigma=0, z''=z'\}]\times X_{b}^{2}(a,\eta)\end{equation}
with parametric dependence on $z$. Integration in $(\tilde x,y'')$ is
pushforward by
a b-fibration onto $X_{b}^{2}(a,\eta)\times [0,1]_{\eta\sqrt t x'}\times
[0,1]_{\sigma}$. From here, we follow precisely the analysis in I-a to show
that (\ref{fourthintlargeone}) is phg conormal on $Q_{0}$.

Finally, consider the integral from $\eta^{2}t$ to $t$. Using $\sigma'$ instead,
we have
\begin{equation}\label{fourthintlargetwo}
\int_{\eta}^{1}\int_{Z}H^{\Omega_{0}}(t(1-(\eta/\sigma')^{2})),z,z'')
F(\sigma',\eta\sqrt t x'', y'',\eta\sqrt t x', y')\ dx''\ dy''\ d\sigma'.
\end{equation}
We proceed analogously to the calculation in region II-a; the first term is
a smooth function of $\sqrt t$ and $\eta/\sigma'$, so by the
pullback theorem, it is phg conormal
on the region of $X_{b}^{3}(\eta\sqrt t,\eta, \sigma')$
where $(\eta\sqrt t/\eta=\sqrt t)$ is bounded and $\eta/\sigma'$ is bounded,
and hence on $X_{b}^{3}(\eta\sqrt t,\eta,\sigma')$.
As in II-a, introduce placeholders $a$ for $\eta\sqrt t$ and $b$ for
$\sigma'$; the analysis is then identical to II-a, and we conclude that
(\ref{fourthintlargetwo}) is phg conormal on $Q_{0}$. This completes
the analysis of the fourth integral, and with it, the proof of the technical
lemma.

\end{document}